\documentclass[11pt]{article}
\usepackage[ruled,vlined,linesnumbered]{algorithm2e}

\usepackage{graphicx,cite}
\usepackage[thinlines]{easytable}
\usepackage{amsmath,amsfonts,amssymb,latexsym}
\usepackage{enumerate}
\usepackage{color,amsthm}
\usepackage[cp1250]{inputenc}
\usepackage{multirow}
\usepackage{hhline}
\usepackage{rotating}
\usepackage{fullpage}

\newtheorem{thm}{Theorem}[section]
\newtheorem{definition}{Definition}[section]
\newtheorem{prp}[thm]{Proposition}

\newtheorem{cor}[thm]{Corollary}

\newtheorem{claim}[thm]{Claim}

\newcommand{\NP}{\ensuremath{\mathsf{NP}}\xspace}

\newcommand{\nof}{$\neg{\exists f}$}

\DeclareMathOperator{\gama}{{\displaystyle\gamma}}
\DeclareMathOperator{\gamat}{{\displaystyle\gamma}\;\!\!_{\it t}}
\DeclareMathOperator{\gamaR}{{\displaystyle\gamma}\;\!\!_{\it R}}

\DeclareMathOperator{\gamatwo}{{\displaystyle\gamma}_2}
\DeclareMathOperator{\gamawtwo}{{\displaystyle\gamma}\;\!\!_{{\it w}\!\!\:2}}
\DeclareMathOperator{\gamaxtwo}{{\displaystyle\gamma}\;\!\!_{{\scriptstyle \times} \! 2}}
\DeclareMathOperator{\rgamaxtwo}{\widetilde{\displaystyle\gamma}\;\!\!_{{\scriptstyle \times} \! 2}}
\DeclareMathOperator{\gamatxtwo}{{\displaystyle\gamma}\;\!\!_{{\it t}\!{\scriptstyle \times} \! 2}}
\DeclareMathOperator{\rgamatxtwo}{\widetilde{\displaystyle\gamma}\;\!\!_{{\it t}\!{\scriptstyle \times} \! 2}}
\DeclareMathOperator{\rgamawtwo}{\widetilde{\displaystyle\gamma}\;\!\!_{{\it w}\!\!\:2}}
\DeclareMathOperator{\rgamatwo}{{\widetilde{\displaystyle\gamma}_2}}
\DeclareMathOperator{\gamasettwo}{{\displaystyle\gamma}\;\!\!_{{\{2\}}}}
\DeclareMathOperator{\rgamasettwo}{\widetilde{\displaystyle\gamma}\;\!\!_{{\{2\}}}}
\DeclareMathOperator{\gamatsettwo}{{\displaystyle\gamma}\;\!\!_{{\it t}{\{2\}}}}
\DeclareMathOperator{\rgamatsettwo}{\widetilde{\displaystyle\gamma}\;\!\!_{{\it t}{\{2\}}}}

\DeclareMathOperator{\Hasse}{{\cal H}}

\title{Domination parameters with number $2$: \\interrelations and algorithmic consequences}
\author{Flavia Bonomo\thanks{Departamento de Computaci\'on, FCEyN, Universidad de Buenos Aires, and CONICET, Argentina. E-mail address: \texttt{fbonomo@dc.uba.ar}}
\and
Bo\v stjan Bre\v sar\thanks{Faculty of Natural Sciences and Mathematics, University of Maribor, Slovenia, and
Institute of Mathematics, Physics and Mechanics, Ljubljana, Slovenia.
E-mail address: \texttt{bostjan.bresar@um.si}}
\and
Luciano N.\ Grippo\thanks{Instituto de Ciencias, Universidad Nacional de General Sarmiento, Los Polvorines, Buenos Aires, Argentina. E-mail address: \texttt{lgrippo@ungs.edu.ar}}
\and
Martin Milani\v c\thanks{University of Primorska, UP IAM, Muzejski trg 2, SI-6000 Koper, Slovenia, and
University of Primorska, UP FAMNIT, Glagolja\v ska 8, SI-6000 Koper, Slovenia. E-mail address: \texttt{martin.milanic@upr.si}.
}
\and
Mart\'in D.\ Safe\thanks{Departamento de Matem\'atica, Universidad Nacional del Sur, Bah\'ia Blanca, Buenos Aires, Argentina, and Instituto de Ciencias, Universidad Nacional de General Sarmiento, Los Polvorines, Buenos Aires, Argentina. E-mail address: \texttt{msafe@uns.edu.ar}}}

\date{\today}

\begin{document}
\maketitle

\vspace{-0.5cm}
\begin{abstract}
In this paper, we study the most basic domination invariants in graphs, in which number~$2$
is intrinsic part of their definitions. We classify them upon three criteria, two of which
give the following previously studied invariants: the weak $2$-domination number,
${\displaystyle\gamma}\,\!\!_{w\!\!\:2}(G)$,
the \hbox{$2$-domination} number, $\gamatwo(G)$, the $\{2\}$-domination number, $\gamasettwo(G)$,
the double domination number, ${\displaystyle\gamma}\,\!\!_{{\scriptstyle \times} \! 2}(G)$, the total $\{2\}$-domination number, $\gamatsettwo(G)$,
and the total double domination number, $\gamatxtwo(G)$, where $G$ is a graph in which a corresponding
invariant is well defined. The third criterion yields rainbow versions of the mentioned six parameters,
one of which has already been well studied,
and three other give new interesting parameters. Together with a special, extensively studied Roman
domination, $\gamaR(G)$, and two classical parameters, the domination number, $\gamma(G)$,
and the total domination number, $\gamma_t(G)$, we consider $13$ domination invariants in graphs.
In the main result of the paper we present sharp upper and lower bounds of each of the invariants in terms
of every other invariant, a large majority of which are new results proven in this paper. As a consequence
of the main theorem we obtain new complexity results regarding the existence of approximation algorithms for the studied invariants,
matched with tight or almost tight inapproximability bounds, which hold even in the class of split graphs.
\end{abstract}



\newpage%%%%%%%%%%%%%%%%%%%%%%%%%%%%%%%%%%%%%%%%%%%%%%%%%%%%%%%%%%%%%%%%%%%%%5
\section{Introduction}

\subsection{Prologue}

A continuously growing interest in the area of graph domination,
which arises from both practical applications and combinatorial
challenges, has made the theory rather incoherent; two monographs
surveying domination theory were published almost twenty years ago
\cite{hhs2,hhs-1998}. Due to a large number of domination-type concepts,
it is not always easy to notice and appreciate some deep results
that capture a broad aspect of the theory. Several results
in domination theory have been in some sense rediscovered, because
an approach that works for one concept can often be used with some slight
adjustment for several other related concepts. We wish to make a step
in the direction of making the situation more transparent,
by classifying some of the most basic domination
invariants, in which number 2 is involved in the definition. We
make a comparison of their values in graphs between each pair of
them, and as a consequence, since the discovered translations
between parameters can be efficiently constructed, a general
approach that joins some algorithmic and complexity issues
on all of these concepts is established. In many cases our results
imply that an algorithm for one invariant gives a good approximation
algorithm for some other invariant; in addition, {strong} inapproximability
results are inferred for {almost all considered} parameters, which hold even
in the class of split graphs. (Let us mention that in \cite{bcf-2007} some connections
between a (different and smaller) group of domination parameters
has been established, yet the main focus was on claw-free graphs.)

\subsection{Classification of parameters}

The central focus of the paper is on several domination invariants of graphs,
which have number 2 appearing in their definition (in particular, vertices
must be dominated twice or using the sum of weights 2), and we can classify
them upon three different criteria. The first criterion
is the set of weights that are allowed to be assigned to vertices,
which can be either $\{0,1,2\}$ or only $\{0,1\}$ (in rainbow
versions, which we will consider in parallel, these weights can be
either $\{\emptyset, \{a\},\{b\},\{a,b\}\}$ or only $\{\emptyset,
\{a\},\{b\}\}$). The second criterion distinguishes three
possibilities with respect to the set of vertices that need to be
dominated,  and at the same time the type of neighborhoods, which
are considered in domination. The possibilities are as follows:
only vertices with weight 0 need to be dominated (`outer
domination'), all vertices need to be dominated and vertices with
a positive weight dominate their closed neighborhoods (`closed
domination'), and finally all vertices need to be dominated and
only open neighborhoods are dominated by vertices with positive
weight (`open domination'). The following table shows the six
concepts that arise from these two criteria, all of which have
already been studied in the literature (in parenthesis a standard
symbol of the corresponding graph invariant is written{\footnote{The
total double domination was also denoted by ${\displaystyle\gamma}\!\:_{\!{\scriptstyle \times} \! 2,t}$ in the literature,
and was also called the double total domination.}):

\bigskip

\begin{center}
\begin{tabular}{|c|c|c|}
  \hline
   & $\{0,1,2\}$ & $\{0,1\}$ \\
   \hline \hline
  outer & weak $2$-domination ($\gamma_{w2}$) & $2$-domination ($\gamma_{2}$)\\
  \hline
  closed & $\{2\}$-domination ($\gamma_{\{2\}}$)  & double domination ($\gamaxtwo$)\\
  \hline
  open &  total $\{2\}$-domination ($\gamma_{t\{2\}}$) & total double domination ($\gamatxtwo$) \\
  \hline
\end{tabular}
\label{tab:criteria}
\end{center}

\bigskip

The third criterion is based on the so-called rainbow variations of these parameters,
and thus distinguishes domination parameters as being rainbow or not. This criterion is
motivated by the concept known as $k$-rainbow domination introduced in \cite{bhr-2008};
in the case $k=2$ the corresponding graph invariant was denoted by $\gamma_{r2}$,
see, e.g.,~\cite{bks-2007}. Note that in this paper the concept will be called rainbow weak $2$-domination,
and the invariant will be denoted by $\rgamawtwo$, suggesting that it is the rainbow counterpart
of the concept of weak $2$-domination,  whose graph invariant is denoted by $\gamawtwo$.
The $k$-rainbow domination (and $2$-rainbow domination, in particular) has been considered
in several papers~\cite{clw-2013,cwx-2010,chh-2014,pps-2012,rad-2011,sly-2014,srt-2013}, and is interesting also because of its strong connection
with the domination of Cartesian products of graphs; in fact, some initial results on the $2$-rainbow
domination number in \cite{hr-2004} were expressed in the terminology of domination of prisms.
In this paper we are mainly concerned with its conceptual features, which initiates several other rainbow
domination parameters. Intuitively speaking they are obtained as follows: weight $0$ is replaced by the label $\emptyset$, weight $1$ by
labels $\{a\}$ and $\{b\}$, and weight $2$ by the label $\{a,b\}$, while the conditions imposed by each parameter
are meaningfully adjusted to the rainbow version. The main difference is that instead of the sum of values of
weights, in a rainbow version one considers the union of labels, and also the condition of having weight $2$ in a neighborhood
corresponds to having label $\{a,b\}$.

Given a graph $G$ its weak $2$-domination number is denoted by ${\displaystyle\gamma}\,\,\!\!_{w\!\!\:2}(G)$,
its \hbox{$2$-domination} number by $\gamatwo(G)$, its $\{2\}$-domination number by $\gamasettwo(G)$, its double domination
number by $\gamaxtwo(G)$, its total $\{2\}$-domination number by $\gamatsettwo(G)$ and its total double domination
number by $\gamatxtwo(G)$.
(We remark that the notion of {\em weak $2$-domination} appeared in the literature also
under the name ``weak $2$-rainbow domination''~\cite{bks-2007}.) By the above reasoning
each of these parameters has its rainbow counter-part, which we will denote in a systematic way, by putting the symbol $\,\,{\widetilde{}}\,\,$
above $\gama$, indicating that we are considering the rainbow version of the known concept.
Two of the parameters among $\rgamatwo(G)$, $\rgamawtwo(G)$, $\rgamasettwo(G)$, $\rgamaxtwo(G)$, $\rgamatsettwo(G)$
and $\rgamatxtwo(G)$ (namely $\rgamasettwo(G)$ and $\rgamatsettwo(G)$) turn out to be
easily expressible by the known graph invariants, and we have thus not studied them any further.
We believe that other four rainbow domination parameters are worth of consideration.

There is yet another well studied domination parameter, which involves number $2$,
but does not directly fit into the above frame. Nevertheless, the so-called {\em Roman domination},
introduced in~\cite{st-1999} (see also~\cite{cdhh-2004,lc-2012}) has been considered in a number of papers,
and is conceptually relevant also to our study. In the condition of the Roman dominating function,
only the vertices with weight $0$ must have in the neighborhood a vertex with weight $2$,
while there is no such restriction for the vertices with weight $1$ and $2$.  Beside Roman domination,
whose parameter in denoted by $\gamma_R$,  we decided to include in our study also the two
classical domination concepts,  i.e.,  the {\em domination} and the {\em total domination},
denoted by $\gamma$ and $\gamma_t$, respectively.
Hence in our main result, see Table~\ref{table-bounds} (on p.~\pageref{table-bounds}),
thirteen domination parameters are mutually compared. To stay within a reasonable length of the paper (and to stay in line with the basic classification
presented in this paper) we do not consider other variations that also involve number 2 in their definitions.
In particular, we do not consider the concepts that arise from basic parameters by imposing additional
restrictions (such as paired domination~\cite{hs-1998}, independent Roman domination~\cite{atr-2012},
exact double domination \cite{ckm-2005}, etc.).

\subsection{Algorithmic complexity}

\vbox{The main result of this paper is the list of the sharp upper and lower bounds for each of the parameters,
expressed in terms of any other parameter. The comparison is not only interesting in its own right,
but also has several consequences regarding algorithmic and complexity properties of the invariants involved.}

For some of the invariants studied in this paper ${\sf NP}$-completeness of their decision problems was known in the literature.
In addition, for $\gama,\gamat,\gamatwo,$ $\gamaxtwo,\gamatxtwo$ it was known that
any {polynomial time} approximation of these values to within a multiplicative factor of $(1-\epsilon)\ln n$
is very unlikely even when restricted to $n$-vertex split graphs: {it would imply ${\sf P} = {\sf NP}$.} (See Section~\ref{sec:algo} for details.)
By using the main result of this paper we are able to infer such theorems about inapproximability in split
graphs for all but three considered invariants. For two of the remaining invariants
(namely, for rainbow $2$-domination, $\rgamatwo$, and rainbow double domination, $\rgamaxtwo$)
we obtain the same result using a direct reduction from the {\sc Set Cover} problem.
The only exception to the inapproximability bounds is the rainbow total double domination number, $\widetilde{\displaystyle\gamma}\!\:_{t\!{\scriptstyle \times} \! 2}$, for which
we prove that there is no polynomially computable function $f$ such that there exists an $f(n)$-approximation algorithm for
this invariant in an  $n$-vertex split graphs for which this parameter is finite, unless
${\sf P} = {\sf NP}$.
We prove this using a reduction from the
${\sf NP}$-complete {\sc Hypergraph $2$-Colorability} problem.

On a positive side, for all of the invariants studied in this paper we prove the existence of approximation algorithms
matching the logarithmic lower bound up to a constant factor, with an obvious exception of
$\widetilde{\displaystyle\gamma}\!\:_{t\!{\scriptstyle \times} \! 2}$ and two other parameters, $\rgamatwo$ and $\rgamaxtwo$,
for which this is still open.

\subsection{Organization of the paper}

{In Section~\ref{sec:definitions} we state the definitions of the parameters studied in this paper as well as some preliminaries on
three covering parameters in graphs, and summarize the definitions in Table~\ref{table-defs}.} In Section~\ref{sec:comparison} we present the main results, expressed in the $13\times 13$ table (Table~\ref{table-bounds}), in which
rows and columns represent the considered domination parameters, and each entry contains the upper bound of the row-parameter with respect to the column-parameter in the family of all graphs for which both parameters are finite. Since the diagonal elements are trivially just the equalities, this means that altogether we have $13\cdot 12=156$ sharp upper bounds between all pairs of parameters. Table~\ref{table-proofs} in the same section gives a road map for deduction of proofs, either by references to results in one of the next sections, or by references to the papers in which the results were proven, or (in many cases) by using transitivity.

%In Section~\ref{sec:covers} we prove a result, by using linear programming, which shows that the fractional edge-cover number coincides with
%the half-integer edge cover number. This result, which may be of independent interest, is then used in the next section to prove the bound of
%$\gamatsettwo$ in terms of $\gamawtwo$ (see~Proposition~\ref{prp:gamma-tset2-gamma-w2}).
In Section~\ref{sec:proofs-bounds} we make the comparison
of the parameters, by proving the upper bounds, if they exist, of parameters expressed as functions of other parameters. We omit the
proofs of most of such bounds that can be found in the literature, as well as of those that follow by transitivity from other bounds in Table~\ref{table-bounds}. Having in mind this optimization of the proofs, we only need to prove 17 propositions in this section.
Then, in Section~\ref{sec:sharp}, we present the values of the parameters in different families of graphs, some showing the sharpness of the bounds
in Table~\ref{table-bounds}, and some other showing that a particular parameter is not bounded by a function of another parameter.

In Section~\ref{sec:algo} we discuss the algorithmic and complexity consequences of the bounds obtained in Section~\ref{sec:comparison}, {proving new lower and upper bounds regarding the (in-)approximability of the corresponding optimization problems, subject to the ${\sf P}\neq {\sf NP}$ assumption.} We combine this with a survey on previously known (in-)approximability and NP-hardness results on these parameters.

\section{Definitions and preliminaries}\label{sec:definitions}

Unless stated otherwise, we consider finite, undirected, simple
graphs. Given a vertex $x\in V(G)$, $N(x)=\{v\in V(G)\mid xv\in
E(G)\}$ denotes its {\em (open) neighborhood}, and $N[x]:=
N(x)\cup \{x\}$ is the {\em closed neighborhood} of $x$. For a
graph $G$ and $X\subseteq V(G)$, we write $N(X)$ for
$(\bigcup_{v\in X}N(v))\setminus X$. As usual, $\Delta(G)$ and
$\delta(G)$ stand for the maximum, resp.~the minimum degree of
vertices in $G$.
%%%%%%%%

Let $f:V(G)\to X$ be a function such that $X$ is either a set of real numbers or a set of finite sets.
For an arbitrary subset $W\subseteq V(G)$,
we denote its {\em weight with respect to $f$} (or just {\em weight} when $f$ is clear from the context) by
$$f(W) = \sum_{w\in W}|f(w)|\,,$$
where the notation $|r|$ denotes either the cardinality of $r$ (if $r$ is a set), or
the absolute value of $r$ (if $r$ is a real number).
For a function $f:V(G)\to X$, where $X$ is an arbitrary set of finite sets,
we denote
$$f_{\cup}(W) = \bigcup_{w\in W}f(w)\,.$$

Next, we present definitions of all the invariants studied in the paper.
Whenever an invariant is not defined for all graphs, i.e.,
if there is a graph $G$ for which no function satisfying the corresponding constraints exists,
we use the convention of stating that the value of the invariant in $G$ is infinite.

\bigskip

\vbox{
\noindent{\bf Domination, total domination.}

%%%%%%%%%%%%%%%%%%%%%%%%%%%%%%%%%%%%%%%%%%%%%%%%%%
% Domination
%%%%%%%%%%%%%%%%%%%%%%%%%%%%%%%%%%%%%%%%%%%%%%%%%%

\begin{definition}
($\boldsymbol{\gama}$, row/column {\bf 1} in
Table~\ref{table-bounds})

Let $G=(V,E)$ be a graph. A {\em dominating function of $G$} is a function $\boldsymbol{f:V\to\{0,1\}}$
such that for all $v\in V(G)$ it holds that
$$\boldsymbol{f(N[v])\ge 1}\,.$$ Equivalently,
$$f(v) = 0~\Longrightarrow~f(N(v))\ge 1\,.$$
The {\em domination number of $G$} is denoted by
$\boldsymbol{\gama(G)}$ and equals the minimum weight $f(V)$ over
all dominating functions $f$ of $G$.

Any set of the form $D = \{v\in V\mid f(v) = 1\}$
where $f$ is a dominating function of $G$
is said to be a {\em dominating set} of $G$.
Note that the minimum size of a dominating set equals $\gama(G)$.
\end{definition}}

Domination number is one of the classical graphs invariants;
together with several of its variations it was surveyed in two monographs~\cite{hhs2,hhs-1998}.

%%%%%%%%%%%%%%%%%%%%%%%%%%%%%%%%%%%%%%%%%%%%%%%%%%
% Total domination
%%%%%%%%%%%%%%%%%%%%%%%%%%%%%%%%%%%%%%%%%%%%%%%%%%

\begin{definition}
($\boldsymbol{\gamat}$, row/column {\bf 2} in
Table~\ref{table-bounds})

Let $G=(V,E)$ be a graph. A {\em total dominating function of $G$} is a function $\boldsymbol{f:V\to\{0,1\}}$
such that for all $v\in V(G)$ it holds that
$$\boldsymbol{f(N(v))\ge 1}\,.$$
The {\em total domination number of $G$} is denoted by
$\boldsymbol{\gamat(G)}$ and equals the minimum weight $f(V)$ over
all total dominating functions $f$ of $G$.

Any set of the form $D = \{v\in V\mid f(v) = 1\}$
where $f$ is a total dominating function of $G$
is said to be a {\em total dominating set} of $G$.
Note that the minimum size of a total dominating set equals $\gamat(G)$.
\end{definition}

Clearly, the total domination number is well-defined (i.e.\ is finite)
in graphs with no isolated vertices.
The recent monograph \cite{HYbook} presents a thorough survey on total domination theory.

\bigskip

\vbox{
\noindent{\bf Weak $2$-domination, rainbow weak $2$-domination.}
%%%%%%%%%%%%%%%%%%%%%%%%%%%%%%%%%%%%%%%%%%%%%%%%%%
% weak 2-domination
%%%%%%%%%%%%%%%%%%%%%%%%%%%%%%%%%%%%%%%%%%%%%%%%%%

\begin{sloppypar}
\begin{definition}($\boldsymbol{\gamawtwo}$, row/column {\bf 3} in
Table~\ref{table-bounds})

Let $G=(V,E)$ be a graph. A {\em weak $2$-dominating function of
$G$} is a function $\boldsymbol{f:V\to\{0,1,2\}}$ such that for
all $v\in V(G)$ it holds that
$$\boldsymbol{f(v) = 0~\Longrightarrow~f(N(v)) \ge 2}\,.$$
The {\em weak $2$-domination number of $G$} is denoted by
$\boldsymbol{\gamawtwo(G)}$ and equals the minimum weight $f(V)$
over all weak $2$-dominating functions $f$ of $G$.
\end{definition}
\end{sloppypar}}

%%%%%%%%%%%%%%%%%%%%%%%%%%%%%%%%%%%%%%%%%%%%%%%%%%
% rainbow weak 2-domination
%%%%%%%%%%%%%%%%%%%%%%%%%%%%%%%%%%%%%%%%%%%%%%%%%%

\begin{sloppypar}
\begin{definition}($\boldsymbol{\rgamawtwo}$, row/column {\bf 9} in
Table~\ref{table-bounds})

Let $G=(V,E)$ be a graph. A {\em rainbow weak $2$-dominating
function of $G$} is a function $\boldsymbol{f:V\to{\cal
P}(\{a,b\})}$ such that for all $v\in V(G)$ it holds that
$$\boldsymbol{f(v) = \emptyset~\Longrightarrow |f_\cup(N(v))| \ge 2}\,.$$
Equivalently, for all $v\in V(G)$ with $f(v) = \emptyset$, it
holds that $f_\cup(N(v)) = \{a,b\}$. The {\em rainbow weak
$2$-domination number of $G$} is denoted by
$\boldsymbol{\rgamawtwo(G)}$ and equals the minimum weight $f(V)$
over all rainbow weak $2$-dominating functions $f$ of $G$.
\end{definition}
\end{sloppypar}

Rainbow weak $2$-domination was introduced less than 10 years ago in~\cite{bhr-2008},
under the name {\em $2$-rainbow domination}; it has already been considered
in a number of papers. Weak $2$-domination was studied in~\cite{bks-2007} with
the aim to give more insight in the (weak) $2$-rainbow domination. It has probably
been known before, although we were unable to find a reference confirming it.
Weak $2$-domination should not be confused with the concept of weak domination, as introduced in~\cite{sala-96}.

\bigskip

\vbox{\noindent{\bf $\{2\}$-domination, rainbow $\{2\}$-domination.}
%%%%%%%%%%%%%%%%%%%%%%%%%%%%%%%%%%%%%%%%%%%%%%%%%%
% {2}-domination
%%%%%%%%%%%%%%%%%%%%%%%%%%%%%%%%%%%%%%%%%%%%%%%%%%

\begin{sloppypar}
\begin{definition}($\boldsymbol{\gamasettwo}$, row/column {\bf 4} in
Table~\ref{table-bounds})

 Let $G=(V,E)$ be a graph. A {\em
$\{2\}$-dominating function of $G$} is a function
$\boldsymbol{f:V\to\{0,1,2\}}$ such that for all $v\in V(G)$ it
holds that $$\boldsymbol{f(N[v]) \ge 2}\,.$$ The {\em
$\{2\}$-domination number of $G$} is denoted by
$\boldsymbol{\gamasettwo(G)}$ and equals the minimum weight $f(V)$
over all $\{2\}$-dominating functions $f$ of $G$.
\end{definition}
\end{sloppypar}}

The concept of $\{2\}$-domination was introduced in 1991~\cite{dhl-1991},
and considered later on in several papers. In particular, several recent papers
consider the variation of Vizing's conjecture on the domination number of Cartesian products
of graphs with respect to this domination invariant, see~\cite{bhk-2007,cmh-15,hl-2009,ns-2013}.

%%%%%%%%%%%%%%%%%%%%%%%%%%%%%%%%%%%%%%%%%%%%%%%%%%
% rainbow {2}-domination
%%%%%%%%%%%%%%%%%%%%%%%%%%%%%%%%%%%%%%%%%%%%%%%%%%

\begin{definition}
Let $G=(V,E)$ be a graph. A {\em rainbow $\{2\}$-dominating
function of $G$} is a function $\boldsymbol{f:V\to{\cal
P}(\{a,b\})}$ such that for all $v\in V(G)$ it holds that
$$\boldsymbol{|f_\cup(N[v])| \ge 2}\,.$$ Equivalently, $f_\cup(N[v]) = \{a,b\}$
holds for all $v\in V(G)$. The {\em rainbow $\{2\}$-domination
number of $G$} is denoted by $\boldsymbol{\rgamasettwo(G)}$ and
equals the minimum weight $f(V)$ over all rainbow
$\{2\}$-dominating functions $f$ of $G$.
\end{definition}

It is easy to see that the rainbow $\{2\}$-domination number is
closely related to the domination number. Indeed, for every graph
$G$, it holds that $\boldsymbol{\rgamasettwo(G) = 2\gama(G)}$.
Hence, we will not discuss this parameter any further in the rest
of the paper, except briefly in Sections~\ref{sec:comparison} and~\ref{sec:algo}
(in Theorems~\ref{thm:inapprox-lower-bound} and~\ref{thm:approx-upper-bound}).

\bigskip

\noindent{\bf Total $\{2\}$-domination, rainbow total $\{2\}$-domination.}
%%%%%%%%%%%%%%%%%%%%%%%%%%%%%%%%%%%%%%%%%%%%%%%%%%
% total {2}-domination
%%%%%%%%%%%%%%%%%%%%%%%%%%%%%%%%%%%%%%%%%%%%%%%%%%

\begin{sloppypar}
\begin{definition}($\boldsymbol{\gamatsettwo}$, row/column {\bf 5} in
Table~\ref{table-bounds})

Let $G=(V,E)$ be a graph. A {\em total $\{2\}$-dominating function
of $G$} is a function $\boldsymbol{f:V\to\{0,1,2\}}$ such that for
all $v\in V(G)$ it holds that $$\boldsymbol{f(N(v)) \ge 2}\,.$$
The {\em total $\{2\}$-domination number of $G$} is denoted by
$\boldsymbol{\gamatsettwo(G)}$ and equals the minimum weight
$f(V)$ over all total $\{2\}$-dominating functions $f$ of $G$.
\end{definition}
\end{sloppypar}

Clearly, the total $\{2\}$-domination number is finite precisely in
graphs with no isolated vertices. While the concept has been known for some time,
see two recent papers on the total $\{k\}$-domination~\cite{asv-2013,lh-2009}.

%%%%%%%%%%%%%%%%%%%%%%%%%%%%%%%%%%%%%%%%%%%%%%%%%%
% rainbow total {2}-domination
%%%%%%%%%%%%%%%%%%%%%%%%%%%%%%%%%%%%%%%%%%%%%%%%%%

\begin{definition}
Let $G=(V,E)$ be a graph. A {\em rainbow total $\{2\}$-dominating
function of $G$} is a function $\boldsymbol{f:V\to{\cal
P}(\{a,b\})}$ such that for all $v\in V(G)$ it holds that
$$\boldsymbol{|f_\cup(N(v))| \ge 2}\,.$$ Equivalently, $f_\cup(N(v)) = \{a,b\}$
holds for all $v\in V(G)$. The {\em rainbow total
$\{2\}$-domination number of $G$} is denoted by
$\boldsymbol{\rgamatsettwo(G)}$ and equals the minimum weight
$f(V)$ over all rainbow total $\{2\}$-dominating functions $f$ of
$G$.
\end{definition}

Similarly as the rainbow $\{2\}$-domination number is related to
the domination number via the relation $\rgamasettwo(G) =
2\gama(G)$, the rainbow total $\{2\}$-domination number is related
to the total domination number via the relation
$\boldsymbol{\rgamatsettwo(G) = 2\gamat(G)}$. Hence, we will not discuss this parameter any further in the rest
of the paper, except briefly in Sections~\ref{sec:comparison} and~\ref{sec:algo} (in Theorems~\ref{thm:inapprox-lower-bound} and~\ref{thm:approx-upper-bound}).

\bigskip

\noindent{\bf $2$-domination, rainbow $2$-domination.}

%%%%%%%%%%%%%%%%%%%%%%%%%%%%%%%%%%%%%%%%%%%%%%%%%%
% 2-domination
%%%%%%%%%%%%%%%%%%%%%%%%%%%%%%%%%%%%%%%%%%%%%%%%%%

\begin{definition}($\boldsymbol{\gamatwo}$, row/column {\bf 6} in
Table~\ref{table-bounds})

Let $G=(V,E)$ be a graph. A {\em $2$-dominating function of $G$}
is a function $\boldsymbol{f:V\to\{0,1\}}$ such that for all $v\in
V(G)$ it holds that
$$\boldsymbol{f(v) = 0~\Longrightarrow~f(N(v)) \ge 2}\,.$$
The {\em $2$-domination number of $G$} is denoted by
$\boldsymbol{\gamatwo(G)}$ and equals the minimum weight $f(V)$
over all $2$-dominating functions $f$ of $G$.

Any set of the form $D = \{v\in V\mid f(v) = 1\}$
where $f$ is a $2$-dominating function of $G$
is said to be a {\em $2$-dominating set} of $G$.
Note that the minimum size of a $2$-dominating set equals $\gamatwo(G)$.
\end{definition}

The concept of $k$-domination (and $2$-domination in particular) was introduced back in 1985~\cite{fj-1985},
and was later studied quite extensively, see some recent papers~\cite{cp-2014,dgh-2011,fhv-2008,hp-2013}.

%%%%%%%%%%%%%%%%%%%%%%%%%%%%%%%%%%%%%%%%%%%%%%%%%%
% rainbow 2-domination
%%%%%%%%%%%%%%%%%%%%%%%%%%%%%%%%%%%%%%%%%%%%%%%%%%

\begin{sloppypar}
\begin{definition}($\boldsymbol{\rgamatwo}$, row/column {\bf 10} in
Table~\ref{table-bounds})

Let $G=(V,E)$ be a graph. A {\em rainbow $2$-dominating function
of $G$} is a function
$\boldsymbol{f:V\to\{\emptyset,\{a\},\{b\}\}}$ such that
$$\boldsymbol{f(v) = \emptyset~\Longrightarrow~|f_\cup(N(v))| \ge 2}\,.$$
%$$\textrm{ for all $v\in V(G)$ such that $f(v) = \emptyset$, it holds that }|f_\cup(N(v))| \ge 2\,.$$
Equivalently, for all $v\in V(G)$ with $f(v) = \emptyset$, it
holds that $f_\cup(N(v)) = \{a,b\}$. The {\em rainbow
$2$-domination number of $G$} is denoted by
$\boldsymbol{\rgamatwo(G)}$ and equals the minimum weight $f(V)$
over all rainbow $2$-dominating functions $f$ of $G$.
\end{definition}
\end{sloppypar}

\bigskip
\vbox{\noindent{\bf Double domination, rainbow double domination.}
%%%%%%%%%%%%%%%%%%%%%%%%%%%%%%%%%%%%%%%%%%%%%%%%%%
% double domination
%%%%%%%%%%%%%%%%%%%%%%%%%%%%%%%%%%%%%%%%%%%%%%%%%%

\begin{sloppypar}
\begin{definition} ($\boldsymbol{\gamaxtwo}$, row/column {\bf 7} in
Table~\ref{table-bounds})

Let $G=(V,E)$ be a graph. A {\em double dominating function of
$G$} is a function $\boldsymbol{f:V\to\{0,1\}}$ such that for all
$v\in V(G)$ it holds that $$\boldsymbol{f(N[v]) \ge 2}\,.$$ The
{\em double domination number of $G$} is denoted by
$\boldsymbol{\gamaxtwo(G)}$ and equals the minimum weight $f(V)$
over all double dominating functions $f$ of $G$.

Any set of the form $D = \{v\in V\mid f(v) = 1\}$
where $f$ is a double dominating function of $G$
is said to be a {\em double dominating set} of $G$.
Note that the minimum size of a double dominating set equals $\gamaxtwo(G)$.
\end{definition}
\end{sloppypar}}

Double domination number is finite in graphs without isolated vertices.
It was introduced in~\cite{hh-2000} (see also~\cite{haha-96}),
and was studied by a number of authors; consider for instance some
recent papers~\cite{bmp-13,dhh-2014,dhv-13,kr-13}.

%%%%%%%%%%%%%%%%%%%%%%%%%%%%%%%%%%%%%%%%%%%%%%%%%%
% rainbow double domination
%%%%%%%%%%%%%%%%%%%%%%%%%%%%%%%%%%%%%%%%%%%%%%%%%%

\begin{sloppypar}
\begin{definition}
($\boldsymbol{\rgamaxtwo}$, row/column {\bf 11} in
Table~\ref{table-bounds})

Let $G=(V,E)$ be a graph.
A {\em rainbow double dominating function of $G$} is a function $\boldsymbol{f:V\to\{\emptyset,\{a\},\{b\}\}}$
such that for all $v\in V(G)$ it holds that
$$\boldsymbol{|f_\cup(N[v])| \ge 2}\,.$$
Equivalently, $f_\cup(N[v]) = \{a,b\}$ holds for all $v\in V(G)$.
The {\em rainbow double domination number of $G$} is denoted by
$\boldsymbol{\rgamaxtwo(G)}$ and equals the minimum weight $f(V)$
over all rainbow double dominating functions $f$ of $G$.
\end{definition}
\end{sloppypar}

\bigskip
Note that every graph without isolated vertices has domatic number at least $2$, which means that it admits a domatic $2$-partition, that is, a partition of its vertex set into two dominating sets. (The domatic number was introduced in a paper from 1970's~\cite{ch-1977}, extensively studied afterwards, and surveyed in~\cite{zel-1998}.) To see this, note that for any maximal independent set $S$ in the graph, the pair $(S,V\setminus S)$ is a domatic $2$-partition. Given a domatic $2$-partition $(A,B)$, setting $f(v) = \{a\}$ for all $v\in A$ and
$f(v) = \{b\}$ for all $v\in B$ results in a rainbow double dominating function of $G$, which shows that
the rainbow double domination number is well defined for all graphs without isolated vertices.

The above observation can be strengthened as follows: The rainbow double domination number of a graph $G$ without isolated vertices equals the so-called {\it disjoint domination number} of $G$, defined in~\cite{MR2500476} as the minimum value of $|A|+|B|$ over all pairs $(A,B)$ of disjoint dominating sets of $G$, and denoted by $\gamma\gamma(G)$. The disjoint domination number was studied in several papers,~\cite{MR2766902,MR2980859,MR2558608,MR2546895,MR2646135}.

\begin{prp}\label{prp:disjoint-domination}
For every graph $G$ without isolated vertices, we have $\rgamaxtwo(G) = \gamma\gamma(G)$.
\end{prp}

\begin{proof}
Suppose that $A$ and $B$ are two disjoint dominating sets in $G$. Then setting
$f(v) = \{a\}$ for all $v\in A$, $f(v) = \{b\}$ for all $v\in B$, and $f(v) = \emptyset$ for all $v \in V\setminus(A\cup B)$ results
in a rainbow double dominating function of $G$ with total weight $|A|+|B|$.
Conversely, if $f:V\to\{\emptyset,\{a\},\{b\}\}$ is a rainbow double dominating function of $G$, then $f^{-1}(\{a\})$ and $f^{-1}(\{b\})$ form a pair of disjoint dominating sets of $G$ of total size equal to the total weight of $f$.
\end{proof}

\vbox{\noindent{\bf Total double domination, rainbow total double domination.}
%%%%%%%%%%%%%%%%%%%%%%%%%%%%%%%%%%%%%%%%%%%%%%%%%%
% total double domination
%%%%%%%%%%%%%%%%%%%%%%%%%%%%%%%%%%%%%%%%%%%%%%%%%%

\begin{definition} ($\boldsymbol{\gamatxtwo}$, row/column {\bf 8} in
Table~\ref{table-bounds})

Let $G=(V,E)$ be a graph. A {\em total double dominating function
of $G$} is a function $\boldsymbol{f:V\to\{0,1\}}$ such that for
all $v\in V(G)$ it holds that $$\boldsymbol{f(N(v)) \ge 2}\,.$$
The {\em total double domination number of $G$} is denoted by
$\boldsymbol{\gamatxtwo(G)}$ and equals the minimum weight $f(V)$
over all total double dominating functions $f$ of $G$.

Any set of the form $D = \{v\in V\mid f(v) = 1\}$
where $f$ is a total double dominating function of $G$
is said to be a {\em total double dominating set} of $G$.
Note that the minimum size of a total double dominating set equals $\gamatxtwo(G)$.
\end{definition}}

The total double domination number is finite precisely in graphs $G$ with $\delta(G)\ge 2$.
The invariant, which is in some papers called {\em double total domination},
was studied for instance in~\cite{hk-2010, hy-2010}.

%%%%%%%%%%%%%%%%%%%%%%%%%%%%%%%%%%%%%%%%%%%%%%%%%%
% rainbow total double domination
%%%%%%%%%%%%%%%%%%%%%%%%%%%%%%%%%%%%%%%%%%%%%%%%%%

\begin{definition} ($\boldsymbol{\rgamatxtwo}$, row/column {\bf 12} in
Table~\ref{table-bounds})

Let $G=(V,E)$ be a graph. A {\em rainbow total double dominating
function of $G$} is a function
$\boldsymbol{f:V\to\{\emptyset,\{a\},\{b\}\}}$ such that for all
$v\in V(G)$ it holds that
$$\boldsymbol{|f_\cup(N(v))| \ge 2}\,.$$
Equivalently, $f_\cup(N(v)) = \{a,b\}$ holds for all $v\in V(G)$.
The {\em rainbow total double domination number of $G$} is denoted
by $\boldsymbol{\rgamatxtwo(G)}$ and equals the minimum weight
$f(V)$ over all rainbow total double dominating functions $f$ of
$G$.
\end{definition}

\bigskip
The {\em total domatic number} of a graph $G$ without isolated vertices is the maximum number of total dominating sets of $G$
that form a partition of its vertex set, cf.~\cite{zel-1988}.
Analogously to the disjoint domination number of a graph, we define the {\it disjoint total domination number} of a graph $G$ as
the minimum value of $|A|+|B|$ over all pairs $(A,B)$ of disjoint total dominating sets of $G$, and
denote it by $\gamma_t\gamma_t(G)$. Note that this parameter is finite if and only if $G$ admits a partition of its vertex set into two total dominating sets, that is, if its total domatic number is at least $2$ (this is the case, for instance, for all $k$-regular graphs with $k\ge 4$~\cite{MR3055232}). A similar parameter for digraphs was recently considered in~\cite{Kulli-2014}.

\begin{prp}\label{prp:rgamatxtwo}
For every graph $G$, we have $\rgamatxtwo(G) = \gamma_t\gamma_t(G)$.
In particular, the rainbow total double domination number of $G$ is finite if and only if
$V(G)$ can be partitioned into two total dominating sets.
\end{prp}

\begin{proof}
Let $f:V\to\{\emptyset,\{a\},\{b\}\}$ be a minimum rainbow total double dominating function of $G$.
Since $f_\cup(N(v)) = \{a,b\}$ for all $v\in V(G)$, the set of vertices $f^{-1}(\{a\})$,
that is, the set of vertices labeled by $\{a\}$, is a total dominating set in $G$, and, similarly, so is $f^{-1}(\{b\})$.
Since these two sets are disjoint, we have $\gamma_t\gamma_t(G)\le \rgamatxtwo(G)$.
Note that in this case, $V(G)$ can be partitioned into two total dominating sets, namely $f^{-1}(\{\emptyset,\{a\}\})$
and $f^{-1}(\{b\})$.

Conversely, suppose that $\gamma_t\gamma_t(G)$ is finite, and take a pair $A$, $B$ of disjoint total dominating sets $A$ and $B$
such that $|A|+|B| = \gamma_t\gamma_t(G)$.
Then, the function $f:V\to\{\{a\},\{b\}\}$, defined by
$f(v) = \{a\}$ for all $v\in A$, $f(v) = \{b\}$ for all $v\in B$, and $f(v) = \emptyset$ for all $v\in V\setminus (A\cup B)$,
is a rainbow total double dominating function of $G$ with $f(V) = \gamma_t\gamma_t(G)$.
This implies that $\rgamatxtwo(G)\le \gamma_t\gamma_t(G)$, and consequently $\rgamatxtwo(G)= \gamma_t\gamma_t(G)$.
The above argument also shows that if $V(G)$ can be partitioned into two total dominating sets, then $G$
has a rainbow total double dominating function.
\end{proof}

\noindent{\bf Roman domination.}

%%%%%%%%%%%%%%%%%%%%%%%%%%%%%%%%%%%%%%%%%%%%%%%%%%
% Roman domination
%%%%%%%%%%%%%%%%%%%%%%%%%%%%%%%%%%%%%%%%%%%%%%%%%%

\begin{sloppypar}
\begin{definition} ($\boldsymbol{\gamaR}$, row/column {\bf 13} in
Table~\ref{table-bounds})

Let $G=(V,E)$ be a graph. A {\em Roman dominating function of $G$} is a function $\boldsymbol{f:V\to\{0,1,2\}}$
such that for all $v\in V(G)$ it holds that
$$\boldsymbol{f(v) = 0~\Longrightarrow~(\exists w \in N(v) \textrm{ such that } f(w) = 2)}\,.$$
The {\em Roman domination number of $G$} is denoted by
$\boldsymbol{\gamaR(G)}$ and equals the minimum weight $f(V)$ over
all Roman dominating functions $f$ of $G$.
\end{definition}
\end{sloppypar}

As already mentioned in the introduction, the concept of Roman domination was introduced
by Stewart in~\cite{st-1999}, see also~\cite{cdhh-2004}. It was studied also in the PhD thesis of Dreyer~\cite{MR2701485} and
in a series of papers, see, e.g.,~\cite{lc-2012,MR2725878,MR2376019,MR3097071,MR2967224} for some recent references.

%%%%%%%%%%%%%%%%%%%%%%%%%%%%%%%%%%%%%%%%%%%%%%%%%%
% rainbow Roman domination
%%%%%%%%%%%%%%%%%%%%%%%%%%%%%%%%%%%%%%%%%%%%%%%%%%

Defining the rainbow Roman domination number in the obvious way does not lead to a new graph parameter:
it coincides with the Roman domination number.

For each of above defined domination parameters, given a weight
function $f$, if the defining condition is satisfied for a vertex
$v\in V(G)$, we say that $v$ is dominated (with respect to $f$).

\medskip
For later use in Section~\ref{sec:proofs-bounds}, we now recall also the definitions
of three covering parameters in graphs.\\

%%%%%%%%%%%%%%%%%%%%%%%%%%%%%%%%%%%%%%%%%%%%%%%%%%
% 2-edge covers and 2-vertex covers
%%%%%%%%%%%%%%%%%%%%%%%%%%%%%%%%%%%%%%%%%%%%%%%%%%

\noindent{\bf Edge covers, $2$-edge covers, and $2$-vertex covers.} \\

An \emph{edge cover of $G$} is a function $f:E\to\{0,1\}$ such
that for all $v\in V$, it holds that
$$\sum_{w \in V, vw \in E} f(vw) \ge 1\,.$$

If $G$ is a graph with no isolated vertices, the {\em edge cover
number of $G$} is denoted by $\rho(G)$ and equals the minimum
weight $f(E)$ over all edge covers $f$ of $G$.

A \emph{2-edge cover of $G$} is a function $f:E\to\{0,1,2\}$ such
that for all $v\in V$, it holds that
$$\sum_{w \in V, vw \in E} f(vw) \ge 2\,.$$

If $G$ is a graph with no isolated vertices, the {\em $2$-edge cover
number of $G$} is denoted by $\rho_2(G)$ and equals the minimum
weight $f(E)$ over all $2$-edge covers $f$ of $G$.

A \emph{2-vertex cover of $G$} is a function $f:V\to\{0,1,2\}$
such that for all $vw\in E$, it holds that
$$f(v)+f(w) \ge 2\,.$$

The {\em $2$-vertex cover number of $G$} is denoted by $\tau_2(G)$
and equals the minimum weight $f(V)$ over all $2$-vertex covers $f$
of $G$.

The following result is a consequence of several works by Gallai
\cite{Gallai1957,Gallai1958,Gallai1958b,Gallai1959} (cf.~\cite[Chapter 30]{Schrijver03}),
and will be used to prove the
bound of $\gamatsettwo$ in terms of $\gamawtwo$ in
Proposition~\ref{prp:gamma-tset2-gamma-w2}.

\begin{thm}\label{lem:rho+tau}
For every graph $G=(V,E)$ with no isolated vertices, $\rho_2(G) +
\tau_2(G) = 2|V|$.
\end{thm}

\begin{table}[h!]
\centering { \small
\renewcommand{\arraystretch}{1.4}
\tabcolsep=0.125cm
\begin{tabular}{|l|c|l|rl|}
  \hline
  Name & Notion & Function & Condition &
\\
\hline\hline

 domination & $\gama$ & $f:V\to\{0,1\}$ & $f(N[v])\geq 1$ & $\forall
 v$ \\ \hline

 total domination & $\gamat$ & $f:V\to\{0,1\}$ & $f(N(v))\geq 1$ & $\forall
 v$ \\ \hline

 weak 2-domination & $\gamawtwo$ & $f:V\to\{0,1,2\}$ & $f(N(v))\geq 2$ & if $f(v) = 0$ \\ \hline

 rainbow weak 2-domination & $\rgamawtwo$ & $f:V\to 2^{\{a,b\}}$ & $|f_\cup(N(v))|\geq 2$ & if
 $f(v) = \emptyset$ \\ \hline

 \{2\}-domination & $\gamasettwo$ & $f:V\to\{0,1,2\}$ & $f(N[v])\geq 2$ & $\forall
 v$ \\ \hline

 rainbow \{2\}-domination & $\rgamasettwo$ & $f:V\to 2^{\{a,b\}}$ & $|f_\cup(N[v])|\geq 2$ & $\forall
 v$ \\ \hline

 total \{2\}-domination & $\gamatsettwo$ & $f:V\to\{0,1,2\}$ & $f(N(v))\geq 2$ & $\forall
 v$ \\ \hline

 rainbow total \{2\}-domination & $\rgamatsettwo$ & $f:V\to 2^{\{a,b\}}$ & $|f_\cup(N(v))|\geq 2$ & $\forall
 v$ \\ \hline

 2-domination & $\gamatwo$ & $f:V\to\{0,1\}$ & $f(N(v))\geq 2$ & if $f(v) = 0$  \\ \hline

 rainbow 2-domination & $\rgamatwo$ & $f:V\to\{\emptyset,\{a\},\{b\}\}$ & $|f_\cup(N(v))|\geq 2$ & if
 $f(v) = \emptyset$ \\ \hline

 double domination & $\gamaxtwo$ & $f:V\to\{0,1\}$ & $f(N[v])\geq 2$ & $\forall
 v$ \\ \hline

 rainbow double domination & $\rgamaxtwo$ & $f:V\to\{\emptyset,\{a\},\{b\}\}$ & $|f_\cup(N[v])|\geq 2$ & $\forall
 v$ \\ \hline

 total double domination & $\gamatxtwo$ & $f:V\to\{0,1\}$ & $f(N(v))\geq 2$ & $\forall
 v$ \\ \hline

 rainbow total double domination & $\rgamatxtwo$ & $f:V\to\{\emptyset,\{a\},\{b\}\}$ & $|f_\cup(N(v))|\geq 2$ & $\forall
 v$ \\ \hline

 Roman domination & $\gamaR$ & $f:V\to\{0,1,2\}$ & $\exists w\sim v: f(w)= 2$ & if $f(v) = 0$  \\ \hline

 edge cover & $\displaystyle\rho$ & $f:E\to\{0,1\}$ & $f(E(v))\geq 1$ & $\forall
 v$ \\ \hline

 2-edge cover & $\displaystyle\rho_2$ & $f:E\to\{0,1,2\}$ & $f(E(v))\geq 2$ & $\forall
 v$ \\ \hline

 vertex cover & $\displaystyle\tau$ & $f:V\to\{0,1\}$ & $f(v)+f(w)\geq 1$ & $\forall
 vw \in E$ \\ \hline

 2-vertex cover & $\displaystyle\tau_2$ & $f:V\to\{0,1,2\}$ & $f(v)+f(w)\geq 2$ & $\forall
 vw \in E$ \\ \hline
\end{tabular}
} \caption{Summary of definitions of the parameters under study.}
\label{table-defs}
\end{table}

\section{Comparison of parameters}\label{sec:comparison}

In this section, we state our main result: the comparison of the values of $13$
domination parameters in graphs for each pair of them (Table~\ref{table-bounds}).
We start with describing, in Fig.~\ref{fig:Hasse}, the Hasse diagram of the relation $\le$ on the $15$ graph parameters defined in Section~\ref{sec:definitions}. Given two of these parameters, say $\rho$ and $\rho'$, we write $\rho\le \rho'$ if and
only if for every graph $G$ for which both $\rho(G)$ and $\rho'(G)$ are well defined, it holds that $\rho(G)\le \rho'(G)$.
The relations represented in this figure will be used often in our
proofs of upper bounds for the parameters in terms of functions of
other parameters, and in the proofs that these bounds are sharp.

\begin{figure}[h!]
\begin{center}
  \includegraphics[width=0.5\textwidth]{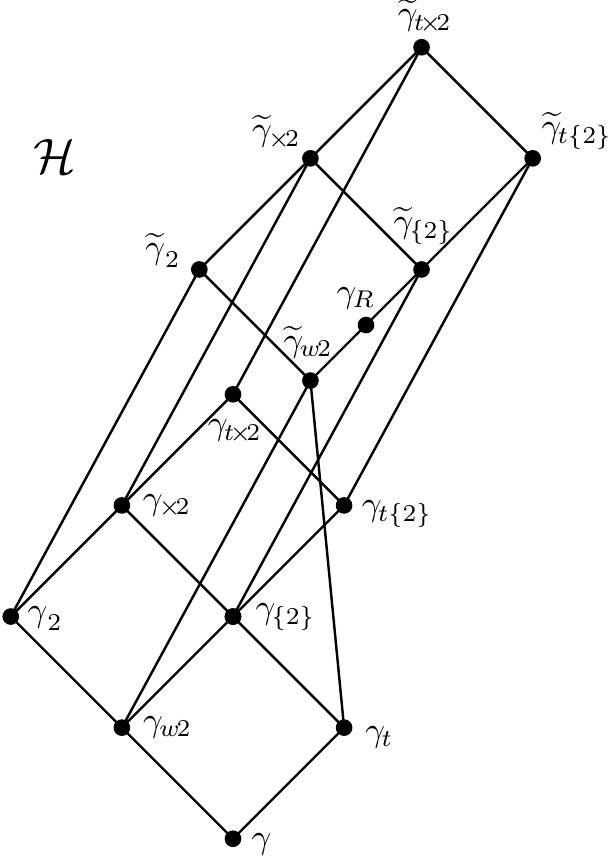}
\caption{Hasse diagram of the relation $\le$ among various
domination parameters.} \label{fig:Hasse}
\end{center}
\end{figure}

\begin{prp}\label{prp:Hasse-relations}
For any two parameters $\rho$ and $\rho'$ in the Hasse diagram on
Fig.~\ref{fig:Hasse},
$\rho$ is below $\rho'$ in the diagram if
and only if $\rho\le \rho'$.
\end{prp}

\begin{proof}
Here, we will only argue the `only if' direction of the proof, that is, if
$\rho$ is below $\rho'$ in the diagram then $\rho\le \rho'$.
The other direction will follow from results in Section~\ref{sec:sharp}.

Clearly, it suffices to verify the statement only for the `covering' pairs $(\rho, \rho')$ in the diagram, that is, pairs such that
$\rho$ is immediately below $\rho'$ in the diagram -- the inequalities for all the remaining pairs follow by transitivity.
The inequality $\gama\leq \gamawtwo$ can be proved by observing that if
$f:V(G)\to\{0,1,2\}$ is a minimum weight weak $2$-dominating function of $G$,
then the set $f^{-1}(\{1,2\})$ is a dominating set of $G$
of total size $|f^{-1}(1)|+|f^{-1}(2)|\le f(V)= \gamawtwo(G)$.
Similarly, if $f:V(G)\to\{0,1,2\}$ is a minimum weight $\{2\}$-dominating function of $G$,
then the set $f^{-1}(\{1,2\})$ is a total dominating set of $G$
of total size $|f^{-1}(1)|+|f^{-1}(2)|\le f(V)= \gamasettwo(G)$. This implies the inequality
$\gamat\le \gamasettwo$. The inequality $\gamma_t\leq{\tilde\gamma}_{w2}$ was
proved in \cite{MR3320720} and the inequality ${\tilde\gamma}_{w2}\leq\gamma_R$ in \cite{wx-2010}.

If $f$ is a dominating function of $G$, then $2f$ is a Roman dominating function of $G$ (cf.~\cite{cdhh-2004,MR1991720}). Hence,
    for every graph $G$, it holds $\gamaR(G) \le 2\gama(G) = \rgamasettwo(G)$, which establishes the relation
$\gamaR\le \rgamasettwo$ in the diagram.

That every rainbow parameter is above its original counterpart
is a direct consequence of definitions. Similarly, all other inequalities represented in the diagram $\mathcal H$
can be easily derived just by looking at the definitions of parameters.
\end{proof}

%%%%%%%%%%%% TABLEs 1 and 2

In Table \ref{table-bounds} we summarize the bounds relating
any two of the $13$ considered parameters, or the fact that there is no bound. We will
give the necessary proofs of upper bounds in Section
\ref{sec:proofs-bounds} and summarize them in Table \ref{table-proofs}.
All the bounds are sharp, as will be shown in Section
\ref{sec:sharp}, which will also contain the proofs of
nonexistence of bounds between certain pairs of parameters.
The families proving the nonexistence of a bound and the examples showing
sharpness are summarized in Tables~\ref{table-families} and~\ref{table-sharp}
(on p.~\pageref{table-families} and~\pageref{table-sharp}), respectively.

%\begin{table}[h!]
\begin{sidewaystable}
\centering { \small
\renewcommand{\arraystretch}{1.4}
\tabcolsep=0.125cm
\begin{tabular}{|c|c||c|c|c|c|c|c|c|c|c|c|c|c|c|c|}
  \hline
\multicolumn{2}{|c||}{\multirow{2}{*}{$\rho\le f(\rho')$}} & 1 & 2 & 3 & 4 & 5 & 6 & 7 & 8 & 9 & 10 & 11 & 12 & 13\\
\hhline{~~-------------}
\multicolumn{2}{|c||}{} &  $\gamma$ & $\gamat$ & $\gamawtwo$ & $\gamasettwo$ & $\gamatsettwo$ & $\gamatwo$ & $\gamaxtwo$ & $\gamatxtwo$   & $\rgamawtwo$ & $\rgamatwo$ & $\rgamaxtwo$ & $\rgamatxtwo$ &  $\gamaR$   \\
\hline\hline
  1 & $\gama$
& $=$ & $\gamat$ & $\gamawtwo$ & $\gamasettwo-1$ & $\gamatsettwo-1$ &
$\gamatwo$ &  $\gamaxtwo-1$ & $\gamatxtwo-1$
  & $\rgamawtwo$ & $\rgamatwo$ & $\frac{1}{2}\rgamaxtwo$ & $\frac{1}{2}\rgamatxtwo$ &
  $\gamaR-1$\,$^*$    \\
\hline

  2 & $\gamat$

& $2\gama$ & $=$ & $\frac{3\gamawtwo-1}{2}$ & $\gamasettwo$ &
$\gamatsettwo-1$ & $\frac{3\gamatwo-1}{2}$ & $\gamaxtwo$ &
$\gamatxtwo-1$
  & $\rgamawtwo$ & $\rgamatwo$ & $\rgamaxtwo$ & $\frac{1}{2}\rgamatxtwo$ &
  $\gamaR$   \\
  \hline

  3 & $\gamawtwo$

&  $2\gama$ & $2\gamat$ & $=$ & $\gamasettwo$ & $\gamatsettwo$ &
$\gamatwo$ & $\gamaxtwo$ & $\gamatxtwo$
  & $\rgamawtwo$ & $\rgamatwo$ & $\rgamaxtwo$ & $\rgamatxtwo$ & $\gamaR$   \\

\hline

  4 & $\gamasettwo$
 & $2\gama$ & $2\gamat$ & $2\gamawtwo-1$\,$^*$    & $=$ & $\gamatsettwo$ & $2\gamatwo-1$\,$^*$     & $\gamaxtwo$ & $\gamatxtwo$ &
  $2\rgamawtwo-1$\,$^*$  & $2\rgamatwo-1$\,$^*$    & $\rgamaxtwo$ & $\rgamatxtwo$ & $2\gamaR-2$\,$^*$    \\
\hline

  5 & $\gamatsettwo$
 & $4\gama$ & $2\gamat$ & $2\gamawtwo$  & $2\gamasettwo$ &  $=$  & $2\gamatwo$ & $2\gamaxtwo$ & $\gamatxtwo$
  & $2\rgamawtwo$ & $2\rgamatwo$ & 2$\rgamaxtwo$ & $\rgamatxtwo$ & $2\gamaR$\\
\hline

  6 & $\gamatwo$

 &  \nof& \nof& \nof & \nof & \nof & $=$ & $\gamaxtwo$ & $\gamatxtwo$
  &\nof &  $\rgamatwo$  & $\rgamaxtwo$ & $\rgamatxtwo$ & \nof\\
\hline

  7 & $\gamaxtwo$

 & \nof&\nof & \nof & \nof & \nof & $2\gamatwo-1$ & $=$  & $\gamatxtwo$
  &\nof & $2\rgamatwo-1$  & $\rgamaxtwo$ & $\rgamatxtwo$ & \nof\\
\hline

  8 & $\gamatxtwo$
 & \nof&\nof & \nof & \nof & \nof & $3\gamatwo-2$ & $2\gamaxtwo-1$  &  $=$
  &\nof  &  $3\rgamatwo-2$  & $2\rgamaxtwo-1$ & $\rgamatxtwo$ & \nof\\

 \hline
9 & $\rgamawtwo$ & $2\gama$ & $2\gamat$  &  $2\gamawtwo-2$ & $2\gamasettwo-2$ &$2\gamatsettwo-2$ & $2\gamatwo-2$ & $2\gamaxtwo-2$ & $2\gamatxtwo-2$ & $=$ &  $\rgamatwo$  & $\rgamaxtwo$ & $\rgamatxtwo$ & $\gamaR$\\
\hline
  10 & $\rgamatwo$
 &  \nof&\nof & \nof & \nof & \nof & \nof & \nof & \nof
  &\nof &  $=$  & $\rgamaxtwo$ & $\rgamatxtwo$ & \nof\\
\hline
  11 & $\rgamaxtwo$
 &  \nof&\nof & \nof & \nof & \nof & \nof &\nof &\nof
  &\nof & $2\rgamatwo$ & $=$   & $\rgamatxtwo$ & \nof\\
\hline
  12 & $\rgamatxtwo$
 &  \nof&\nof & \nof & \nof & \nof &\nof  & \nof&\nof
  &\nof &  \nof&\nof & $=$    & \nof\\
\hline
  13 & $\gamaR$ & $2\gama$ & $2\gamat$ & $2\gamawtwo-1$ & $2\gamasettwo-2$ & $2\gamatsettwo-2$ & $2\gamatwo-1$
   & $2\gamaxtwo-2$& $2\gamatxtwo-2$& $\frac{3}{2}\rgamawtwo$ & $\frac{3}{2}\rgamatwo$ & $\rgamaxtwo$ & $\rgamatxtwo$& $=$ \\
  \hline
\end{tabular}
} \caption{Upper bounds for several domination parameters in terms
of others. The entry in row indexed by parameter $\rho$ and in
column indexed by parameter $\rho'$ represents either that $\rho$
is not bounded from above by any function of $\rho'$ (in this case
the entry is~``\nof''), or gives an upper bound for $\rho$ in
terms of a function of $\rho'$. The bound $\rho\le f(\rho')$
should be interpreted so that it holds for graphs $G$ such that
both $\rho(G)$ and $\rho'(G)$ are defined. For instance, the entry
in row $\gamat$ and column $\gama$ is $2\gama$, indicating that
for every graph $G$ without isolated vertices, it holds that
$\gamat(G)\le 2\gama(G)$. Bounds of the form $\rho\le f(\rho')$
that are marked with an asterisk are only valid for graphs with at
least one edge. The (sharp) bounds for graphs with no edges are,
respectively, $\gama = \gamaR$, $\gamasettwo = 2\gamawtwo =
2\gamatwo = 2\rgamawtwo = 2\rgamatwo = 2\gamaR$, and can be
easily verified. The bound in entry $(9,3)$ holds if $G\neq K_1$.
The correctness of the entries in the table is proved or referenced in Section~\ref{sec:proofs-bounds}
(proofs of upper bounds) and in Section~\ref{sec:sharp} (proofs of nonexistence of upper bounds),
while the references to proofs are summarized in Table~\ref{table-proofs}.
All the bounds are sharp. The examples for sharpness and for the unboundedness
are summarized in Table~\ref{table-sharp}.} \label{table-bounds}
%\end{table}
\end{sidewaystable}

\begin{sidewaystable}
\centering { \small
\renewcommand{\arraystretch}{1.4}
\tabcolsep=0.125cm
\begin{tabular}{|c|c||c|c|c|c|c|c|c|c|c|c|c|c|c|c|}
  \hline
\multicolumn{2}{|c||}{\multirow{2}{*}{$\rho\le f(\rho')$}} & 1 & 2 & 3 & 4 & 5 & 6 & 7 & 8 & 9 & 10 & 11 & 12 & 13\\
\hhline{~~-------------}
\multicolumn{2}{|c||}{} &  $\gamma$ & $\gamat$ & $\gamawtwo$ & $\gamasettwo$ & $\gamatsettwo$ & $\gamatwo$ & $\gamaxtwo$ & $\gamatxtwo$   & $\rgamawtwo$ & $\rgamatwo$ & $\rgamaxtwo$ & $\rgamatxtwo$ &  $\gamaR$   \\
\hline\hline
  1 & $\gama$
& $=$ & $\Hasse$ & $\Hasse$ & \ref{prp:gamma-gamma-set2}  & T4 &
$\Hasse$ & T4 & T2
  & $\Hasse$ & $\Hasse$ & \ref{prp:gamma-rgamma-x2} & T11 &
  \ref{prp:gamma-gammaR}   \\
\hline

  2 & $\gamat$

& $\Hasse$~\cite{hs-1998}
  & $=$ & \ref{prp:gamma-t-gamma-w2} & $\Hasse$ &
    \ref{prp:gamma-t-gamma-tset2} & T3 & $\Hasse$ &
T5
  & $\Hasse$~\cite{MR3320720}
  & $\Hasse$  & $\Hasse$ & \ref{prp:gamma-t-rgamma-tx2} &
  $\Hasse$~\cite{MR3220299}
  \\
  \hline

  3 & $\gamawtwo$

&  T4 & T1 & $=$ & $\Hasse$ & $\Hasse$ & $\Hasse$ & $\Hasse$ & $\Hasse$
  & $\Hasse$ & $\Hasse$ & $\Hasse$ & $\Hasse$ & $\Hasse$   \\

\hline

  4 & $\gamasettwo$
 & \ref{prp:gamma-set2-gamma} & T1 & \ref{prp:gamma-set2-gamma-w2}    & $=$ & $\Hasse$ & T3 &
 $\Hasse$ & $\Hasse$ &
  T3  & T3   & $\Hasse$ & $\Hasse$ & T1    \\
\hline

  5 & $\gamatsettwo$
 & T2 & \ref{prp:gamma-tset2-gamma-t} & \ref{prp:gamma-tset2-gamma-w2}  & T2 &  $=$  & T3 & T3 & $\Hasse$
  & T3 & T3 & T3 & $\Hasse$ & T3\\
\hline

  6 & $\gamatwo$

 &  T5 & T5 & T5 & T5 & \ref{prp:unb-gamatwo-gamatsettwo} & $=$ & $\Hasse$ & $\Hasse$
  & T5 &  $\Hasse$  & $\Hasse$ & $\Hasse$ & T5 \\
\hline

  7 & $\gamaxtwo$

 & T5 & T5 & T5 & T5 & T6 & \ref{prp:gamma-x2-gamma-2} & $=$  & $\Hasse$
  & T5 & T6  & $\Hasse$ & $\Hasse$ & T5 \\
\hline

  8 & $\gamatxtwo$
 & T5 & T5 & T5 & T5 & \ref{prp:unb-gamatxtwo-gamatsettwo} & \ref{prp:gamma-tx2-gamma-2} & \ref{prp:gamatxtwo-gamaxtwo}  &  $=$
  & T5  &  T6  & T7 & $\Hasse$ & T5 \\

 \hline
9 & $\rgamawtwo$ & T13 & T1  &  \ref{prp:rgamma-w2-gamma-w2} & T3
& T3 & T3 & T3 & T3 & $=$ &  $\Hasse$  & $\Hasse$ & $\Hasse$ &
  $\Hasse$~\cite{wx-2010}
  \\
\hline

  10 & $\rgamatwo$
 &  T5 & T5 & T5 & T5 & T6& T8 & T8 & \ref{prp:unb-rgamatwo-gamatxtwo}
  & T5 &  $=$  & $\Hasse$ & $\Hasse$ & T5 \\
\hline
  11 & $\rgamaxtwo$
 &  T5 & T5& T5 & T5 & T6 & T8 & T8 & T10
  & T5 & \ref{prp:rgamma-x2-rgamma-2} & $=$   & $\Hasse$ & T5 \\
\hline
  12 & $\rgamatxtwo$
 &  T11 & T11 &  T11 & T11  &  T11 & T11 &  T11 & T11
  & T11  &  T11 & \ref{prp:unb-rgamatxtwo-rgamaxtwo} & $=$    & T11 \\
\hline
  13 & $\gamaR$
  & \cite{cdhh-2004,MR1991720}
  & T1 & \ref{prp:gamma-R-gamma-w-2} & T1 & T4 & T3
   & T4 & T5 & \cite{MR3097710,FF-2012} & T9 & $\Hasse$ & $\Hasse$ & $=$ \\
  \hline
\end{tabular}
} \caption{Summary of proofs for the entries of Table
\ref{table-bounds}. Entries labeled $\Hasse$ refer to
results that are already stated in the corresponding Hasse
diagram of Fig.~\ref{fig:Hasse}. Entries labeled
``T$k$'' are proved by transitivity through parameter in
row/column $k$. Details regarding the use of
transitivity in proving nonexistence of bounds will be explained
in Section~\ref{ss:unbound}.
Entries labeled with numbers refer to
propositions that directly prove the bound
(or nonexistence of a bound).
Finally, for some entries we give references to papers
where a proof of the corresponding bound can be found.}
\label{table-proofs}
\end{sidewaystable}

We can also consider the following weaker version of the $\le$ relation
on the $15$ graph parameters defined in Section~\ref{sec:definitions}.
The relation $\preceq$ is defined by $\rho\preceq \rho'$ if and
only if there exists a function $f$ such that
for every graph $G$ for which both $\rho(G)$ and
$\rho'(G)$ are well defined, it holds that $\rho(G)\le f(\rho'(G))$.
This relation is reflexive and transitive, but not antisymmetric.
It induces an equivalence relation $\approx$, defined by $\rho\approx \rho'$
if and only if $\rho\preceq \rho'$ and $\rho'\preceq \rho$.
The results summarized in Table~\ref{table-bounds} imply that
the relation $\approx$ has exactly four equivalence classes, which
are linearly ordered by the quotient partial order obtained from $\preceq$ by
collapsing each equivalence class of $\approx$ into a single element.
See Fig.~\ref{fig:Hasse-quotient} for a depiction of these four equivalence classes
and the Hasse diagram of the corresponding linear order.

\begin{figure}[h!]
\begin{center}
  \includegraphics[width=0.3\textwidth]{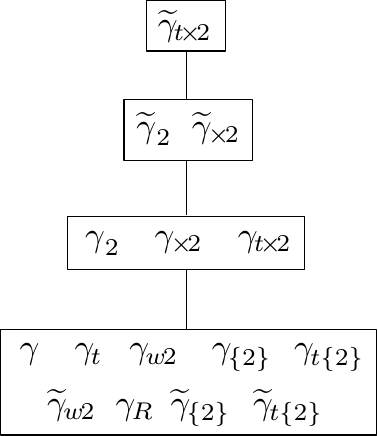}
\caption{The Hasse diagram representing the preorder $\preceq$ on the considered domination parameters.} \label{fig:Hasse-quotient}
\end{center}
\end{figure}

%\section{Proofs of bounds}
%(organized in subsections columnwise, i.e. bounds in terms of a given parameter)

\section{Proofs of upper bounds in Table~\ref{table-bounds}}\label{sec:proofs-bounds}

%%%%%%%%%%%%%%%%%%%%%%%%%%%%%%%%%%% BOUNDS HASSE %%%%%%%%%%%%%%%%%%%%%%%

In this section we prove the upper bounds from Table~\ref{table-bounds}.
All the bounds are sharp, which will be demonstrated in Section~\ref{sec:sharp}.
In the first subsection we concentrate on the bounds that follow from the
Hasse diagram in~Fig.~\ref{fig:Hasse}, while in the second subsection we
give explicit proofs of the remaining bounds. Note that the bounds in the entries
of Table~\ref{table-bounds} that are not proven directly in this section,
follow by transitivity from other bounds, as shown in Table~\ref{table-proofs}.

\subsection{Upper bounds following from the Hasse diagram on Fig.~\ref{fig:Hasse}}

The first proposition of this subsection is a direct consequence of Proposition~\ref{prp:Hasse-relations}.

\begin{sloppypar}
\begin{prp}\label{prp:upper_bounds_Hasse}
The upper bounds indicated by the following entries in Table~\ref{table-bounds}
%(following from Hasse diagram $\Hasse$ on Fig.~\ref{fig:Hasse})
are correct:
$(1,2)$, $(1,3)$,
$(1,6)$,
$(1,9)$,
$(1,10)$
$(2,4)$
$(2,7)$
$(2,9)$,
$(2,10)$
$(2,11)$
$(2,13)$,
$(3,4)$,
$(3,5)$, $(3,6)$, $(3,7)$, $(3,8)$, $(3,9)$, $(3,10)$, $(3,11)$,
$(3,12)$, $(3,13)$,
$(4,5)$, $(4,7)$, $(4,8)$, $(4,11)$, $(4,12)$,
$(5,8)$, $(5,12)$,
$(6,7)$, $(6,8)$, $(6,10)$, $(6,11)$, $(6,12)$,
$(7,8)$, $(7,11)$, $(7,12)$,
$(8,12)$,
$(9,10)$, $(9,11)$, $(9,12)$, $(9,13)$,
$(10,11)$, $(10,12)$,
$(11,12)$,
$(13,11)$,
$(13,12)$.\qed
\end{prp}
\end{sloppypar}
%
%\begin{proof}
%{The upper bounds indicated by all the stated entries in Table~\ref{table-bounds}
%follow immediately from the Hasse diagram $\Hasse$ in Fig.~\ref{fig:Hasse}, except the bound
%indicated by entry $(1,13)$, namely $\gama(G)\le \gamaR(G)-1$.}
%To see that $\gama(G)\le \gamaR(G)-1$ for every graph $G$ with
%at least one edge, let $f$ be a minimum weight Roman dominating
%function of $G$. Let $V_i = \{v\in V(G)\mid f(v) = i\}\,,$ for
%$i\in \{0,1,2\}$. Suppose first that $V_2=\emptyset$. Then also
%$V_0=\emptyset$, and clearly $\gamaR(G)=f(V(G))=|V_1|=|V(G)|\ge
%\gama(G)+1$, where the last inequality follows from the fact that
%$G$ has at least one edge. On the other hand, if
%$V_2\ne\emptyset$, consider the set  $D=V_1\cup V_2$. Clearly, $D$
%is a dominating set of $G$. Now, $$\gama(G)\le |D|=|V_1|+|V_2|\le
%|V_1|+2|V_2|-1=f(V(G))-1=\gamaR(G)-1\,.%\hfill\qedhere
%$$
%\end{proof}

%%%%%%%%%%%%%%%%%%%%%%% %%%%%%%%%%%%%%%%

\subsection{Other upper bounds in Table~\ref{table-bounds}}

In this subsection the remaining upper bounds are proved.
To ease an examination the bounds are numbered by the ordered pairs $(r,c)$,
where $r$ stands for the row and $c$ for the column in the table.
Since a proof of the bound labeled by a pair $(r,c)$ proceeds by taking an optimal
solution for the parameter indexed by column $c$ and modifying it into a feasible solution
for the parameter indexed by row $r$, we group together the proofs of bounds sharing the same column coordinate.
We proceed in increasing order of columns and, within the same column,
in increasing order of rows.

The proofs of Propositions~\ref{prp:gamma-set2-gamma-w2} and~\ref{prp:gamma-tset2-gamma-w2} below
make use of two classical results due to Gallai: one on the structure of minimum edge covers in graphs,
and one on the relation between the $2$-edge-cover and the $2$-vertex cover numbers
of a graph (Theorem~\ref{lem:rho+tau}), respectively.

\subsubsection*{Bound (4,1)}

\begin{prp}\label{prp:gamma-set2-gamma}
For every graph $G$, $\gamasettwo(G)\le 2\gama(G)$.
\end{prp}

\begin{proof}
The inequality $\gamasettwo(G)\le 2\gama(G)$  follows from the
fact that if $f:V(G)\to\{0,1\}$ is a minimum weight dominating
function of $G$, then $g = 2f:V(G)\to \{0,1,2\}$ is a
$\{2\}$-dominating function of $G$ of weight exactly $2\gama(G)$.
\end{proof}

\subsubsection*{Bound (5,2)}

\begin{prp}\label{prp:gamma-tset2-gamma-t}
For every graph $G$ without isolated vertices,
$\gamatsettwo(G)\le
2\gamat(G)$.
\end{prp}

\begin{proof}
The inequality $\gamatsettwo(G)\le
2\gamat(G)$ follows from the
fact that if $f:V(G)\to\{0,1\}$ is a minimum weight total dominating
function of $G$, then $g = 2f:V(G)\to \{0,1,2\}$ is a
total $\{2\}$-dominating function of $G$ of weight exactly $2\gamat(G)$.
\end{proof}

\subsubsection*{Bound (2,3)}

\begin{prp}\label{prp:gamma-t-gamma-w2}
For every graph $G$ without isolated vertices, $\gamat(G)\le
\frac{3\gamawtwo(G)-1}{2}$.
\end{prp}

\begin{proof}
Let $G$ be a graph without isolated vertices, and let
$f:V(G)\to\{0,1,2\}$ be a minimum weight weak $2$-dominating
function of $G$.
Let us define the following subsets of $V(G)$.
Let $V_i = \{v\in V(G)\mid f(v) = i\}\,,$ for $i\in \{0,1,2\}$.
%Note that by minimality of $f$, every vertex in $V_2$ has a neighbor in $V_0$
%(otherwise we could relabel it with $1$ and obtain a
%weak $2$-dominating function of $G$ of smaller weight).
Let $S$ be a maximal set of vertices in $V_0$ such
that their neighborhoods intersected with $V_1$ are nonempty and
pairwise disjoint. Let $D_0$ denote the set of vertices in $V_1$
that have no neighbor in $S$. For every vertex $v\in D_0\cup V_2$,
choose a vertex $v'$ adjacent to it. (Notice that such a vertex
exists, since $G$ has no isolated vertices.) Let $D_0' = \{v'\mid
v\in D_0\}$ and $V_2' = \{v'\mid v\in V_2\}$.

We claim that the set $D' = (V_1\setminus D_0)\cup D_0'\cup S\cup
V_2\cup V_2'$ is a total dominating set of $G$. If $v\in
V_1\setminus D_0$, then $v$ has a neighbor in $S$. If $v\in D_0$,
then $v$ has a neighbor in $D_0'$. If $v\in V_2$, then $v$ has a
neighbor in $V_2'$. If $v\in V_0$, then $v$ has either a neighbor
in $V_2$, or it has at least two neighbors in $V_1$, and thus by
the  definition of $S$, $v$ has a neighbor in $N(S)\cap
V_1\subseteq V_1\setminus D_0$.

Since every vertex in $S$ has at least two neighbors in $V_1$, we
have $|N(S)\cap V_1|\ge 2|S|$. Therefore, we can bound the size of
$D'$ from above as follows:
\begin{eqnarray}\label{eq1}
 |D'| = (|V_1\setminus D_0|+|D_0'|)+|S|+(|V_2|+|V_2'|) \le
   |V_1|+\frac{|N(S)\cap V_1|}{2}+2|V_2|\,.
\end{eqnarray}
%\begin{eqnarray}\label{eq1}
%\nolabel |D'| &=& (|V_1\setminus D_0|+|D_0'|)+|S|+)|V_2|+|V_2'|)\\
%   &\le&
%   |V_1|+\frac{|N(S)\cap V_1|}{2}+2|V_2|\,.
%\end{eqnarray}

If $D_0 \neq\emptyset$, then $|N(S)\cap V_1|\le |V_1|-1$ and hence
by~\eqref{eq1}, we have
$$|D'| \le  |V_1|+\frac{|V_1|-1}{2}+2|V_2|
\le \frac{3(|V_1|+2|V_2|)-1}{2} = \frac{3\gamawtwo(G)-1}{2}\,.$$

If $V_2 \neq \emptyset$, then $2|V_2|\le 3|V_2|-1$ and
by~\eqref{eq1} we obtain
$$|D'| \le |V_1|+\frac{|N(S)\cap V_1|}{2}+3|V_2|-1<
\frac{3(|V_1|+2|V_2|)-1}{2} = \frac{3\gamawtwo(G)-1}{2}\,.$$

Finally, if $D_0 = \emptyset$ and $V_2 = \emptyset$, then $D' =
V_1\cup S$, and we can obtain a smaller total dominating set $D''$
by deleting from $D'$ an arbitrary vertex of $V_1$. Indeed, every
vertex in $V_0$ has at least two neighbors in $V_1$, and hence it
has a neighbor in $D''$. We can bound the size of $D''$ from above
as follows:
$$|D''| = |V_1|+|S|-1\le |V_1|+\frac{|N(S)\cap V_1|}{2}-1\le \frac{3|V_1|}{2}-1
<\frac{3\gamawtwo(G)-1}{2}\,.$$

In either case, we obtain $\gamat(G)\le
\frac{3\gamawtwo(G)-1}{2}$.

\end{proof}

\subsubsection*{Bound (4,3)}

\begin{prp}\label{prp:gamma-set2-gamma-w2}
For every graph $G$ with at least one edge, $\gamasettwo(G)\le
2\gamawtwo(G)-1$.
\end{prp}

\begin{proof}
Let $G$ be a graph with at least one edge.

Let $f:V(G)\to\{0,1,2\}$ be a minimum weight weak $2$-dominating
function of $G$. Let $V_i = \{v\in V(G)\mid f(v) = i\}\,,$ for
$i\in \{0,1,2\}$.

We first deal with the case when $V_2 \ne\emptyset$. The function
$g:V(G)\to\{0,1,2\}$ that agrees with $f$ on all vertices except
on the vertices with $f$-value equal to $1$, each of which $g$
maps to $2$, is a $\{2\}$-dominating function of $G$. Since
$$g(V(G))  = 2|V_1|+2|V_2|\le 2(|V_1|+2|V_2|-1) = 2f(V(G))-2
=2\gamawtwo(G)-2\,,$$ we have the inequality $\gamasettwo(G)\le
2\gamawtwo(G)-2$ in this case.

Assume now that $V_2 = \emptyset$. Note that the set $I$ of
isolated vertices in $G$ is a subset of $V_1$, and let $V_1' =
V_1\setminus I$.

We may assume that every vertex $u$ in $V_1'$ has a neighbor in
$V_0$. Otherwise, if all neighbors of $u$ are in $V_1'$, then we
can obtain a weak $2$-dominating function of $G$ of the same
weight as $f$ by moving the weight $1$ from $u$ to one of its
neighbors, and the previous case ($V_2\neq \emptyset$) applies.

Now, let $H$ be the graph with vertex set $V_1'$ in which two
vertices $u$ and $v$ are adjacent if and only if they have a
common neighbor $n(u,v)$ in $V_0$.
Let $h:E(H)\to\{0,1\}$ be a minimum weight edge cover of $H$,
let $C= \{e\in E(H)\mid h(e) = 1\}$ be the support of $h$, and let $N = \{n(u,v)\mid uv\in C\}$.
Note that $|N|=|C|\le |V_1'|-1$ (recall that a minimum edge cover
induces a spanning forest of stars~\cite{Gallai1959}). Consider the
function $g:V(G)\to\{0,1,2\}$, defined as follows:
$$g(v)  =\left\{
  \begin{array}{ll}
    2, & \hbox{if $v\in I$;} \\
    1, & \hbox{if $v\in N\cup V_1'$;} \\
    0, & \hbox{otherwise.}
  \end{array}
\right.$$ Then, $g$ is a $\{2\}$-dominating function of $G$: if
$v\in I$, then clearly $g(N[v]) = 2$; if $v\in N\cup V_1'$, then
$v$ has a neighbor in $N\cup V_1'$, and hence $g(N[v]) \ge 2$; if
$v\in V_0\setminus N$, then $v$ has at least two neighbors in
$V_1'$, and again $g(N[v]) \ge 2$ holds. Since $$g(V(G))  =
2|I|+|V_1'| +|N| \le 2|I|+|V_1'| + |V_1'|-1 = 2|V_1|-1 =
2f(V(G))-1 =2\gamawtwo(G)-1\,,$$ we have the desired inequality
$\gamasettwo(G)\le 2\gamawtwo(G)-1$.
\end{proof}

\subsubsection*{Bound (5,3)}

\begin{prp}\label{prp:gamma-tset2-gamma-w2}
For every graph $G$ with no isolated vertices, $\gamatsettwo(G)\le
2\gamawtwo(G)$.
\end{prp}
\begin{proof}
Let $f:V\to\{0,1,2\}$ be a minimum weight weak $2$-dominating
function of a graph $G=(V,E)$ with no isolated vertices. In
the proof we will construct a total $\{2\}$-dominating function
$g$ of $G$ with weight less or equal to $2f(V)$, yielding the
bound $\gamatsettwo(G)\le 2\gamawtwo(G)$.

Let $V_i = \{v\in V(G)\mid f(v) = i\}\,,$ for $i\in \{0,1,2\}$.
Note that for any vertex $x$ in $V_0$ we already have $f(N(x))\ge
2$, which is the condition imposed on vertices in the total
$\{2\}$-dominating set. On the other hand, $f(N(y))$ can be less
than 2 for vertices $y\in V_1\cup V_2$. Note that as $f$ is
minimum, each $y\in V_2$ is adjacent to a vertex in $V_0$. Suppose
that $V_2\neq \emptyset$ and let $Y$ be a minimum set of vertices
from $V_0$ that dominate all vertices from $V_2$; i.e., $V_2\subset
N(Y)$ and $Y$ is a smallest possible subset of $V_0$ with this
property. Clearly, $|V_2|\ge |Y|$. Now, let $f_1:V\to\{0,1,2\}$ be
the function obtained from $f$
by setting $f_1(y)=2$ for all
$y\in Y$ (changing $f$ only in vertices of $Y$). Note that for any
vertex $x\in V_0\cup V_2$, we have $f_1(N(x))\ge 2$. Moreover,
when restricted to the subgraph $G_1$ of $G$ induced by $V_2\cup
N(V_2)$, $f_1$ is a total $\{2\}$-dominating function of $G_1$
such that $f_1(V(G_1))\le 2f(V(G_1))$. Thus it suffices to
consider the remainder of the graph, i.e., $G-V(G_1)$; we remark
that the function $g$, which we are constructing, coincides with
$f_1$ on $V(G_1)$.

Consider the set $Z$ of vertices $z\in V_1$ that are adjacent to
some other vertex in $V_1$ (in the case when a vertex $z$ from
$V_1$ is adjacent to a vertex in $V_2$, then it is already in
$G_1$, with $f_1(N(z))\ge 2$, so this case need not be considered
any more). If $Z\neq\emptyset$, then let $f_2:V\to\{0,1,2\}$ be the
function obtained from $f_1$ only by increasing the value of all
vertices $z$ from $V_1$ that have a neighbor in $V_1$, by
setting $f_2(z)=2$. Denote by $G_2$ the subgraph of $G$, induced
by $Z\cup N(Z)$. Clearly, $f_2(N(z))\ge 2$ for all vertices
$z \in Z\cup N(Z)$. Moreover, when restricted to $G_2$, $f_2$ is a total
$\{2\}$-dominating function of $G_2$ such that $f_2(V(G_2))\le
2f(V(G_2))$. Thus it suffices to consider the remainder of the
graph, i.e., $G-(V(G_1)\cup V(G_2))$; we remark that the function
$g$, which we are constructing, coincides with $f_2$ on $V(G_1\cup
G_2)$.

Let $x\in V_1$, such that $x$ is not adjacent to any other vertex
in $V_1\cup V_2$ (thus $x$ lies in $G-(V(G_1)\cup V(G_2))$). Then
$x$ has a neighbor in $V_0$, since $G$ has no isolated vertices.
It is possible that all neighbors of $x$ are in $G_1\cup G_2$. Let
$G_3$ be the subgraph of $G$, induced by all vertices $x$ in $V_1$
and not in $G_1\cup G_2$, such that all their neighbors are in
$G_1\cup G_2$. To each vertex of $G_3$, we set $f_3(x)=0$, and to
an arbitrary neighbor $y\in V_0$ of $x$, we set $f_3(y)=2$ (and
$f_3(u)=f_2(u)$ for all other vertices of $G$). Note that
$f_3(N(x))\ge 2$ for any $x\in G_3$, and $f_3$ restricted to
$V(G_1)\cup V(G_2)\cup V(G_3)$ is a total $\{2\}$-dominating
function of the subgraph induced by $V(G_1)\cup V(G_2)\cup
V(G_3)$. In addition, $f_3(V(G_1)\cup V(G_2)\cup V(G_3))\le
2f(V(G_1)\cup V(G_2)\cup V(G_3))$; we remark that the function
$g$, which we are constructing, coincides with $f_3$ on $V(G_1)\cup
V(G_2)\cup V(G_3)$.

Denote by $H$ the remainder of the graph, i.e., $H=G-(V(G_1)\cup
V(G_2)\cup V(G_3))$. Note that each vertex in $H$ from $V_1$ has
at least one neighbor in $V_0\cap V(H)$, and also each vertex from
$V_0\cap V(H)$ has at least two neighbors from $V_1\cap V(H)$ (the
latter is because $f$ is a weak $2$-dominating function, and
vertices from $V_0\cap V(H)$ are not adjacent to any vertex from
$V_2$ nor to any of the vertices of $V_1$ that were settled in the
previous cases). For each $x\in V(H)$ with $f(x)=0$, choose
arbitrarily any of its two neighbors $y,z$ in $V_1$, and delete
all other edges between $x$ and its neighbors in $(V_0\cup
V_1)\setminus\{y,z\}$; call the resulting graph $H'$. Clearly,
$H'$ is a spanning subgraph of $H$, in which each vertex in $V_0$
has degree exactly 2, while vertices from $V_1$ can have an
arbitrary degree, including 0. Remove all the isolated vertices
from $H'$ to obtain the graph $H''$ (we remark that the isolated
vertices will be settled at the end of the proof). Let $K$ be an
arbitrary connected component of $H''$. Since vertices from $V_0$
in $K$ have degree 2 and are adjacent to two vertices from $V_1$,
$K$ is a subdivision of a graph $K'$, whose vertices correspond to
vertices of $V_1\cap K$, and edges in $K'$ correspond to vertices
in $V_0\cap K$.

Now, a function $h':V(K')\cup E(K') \to\{0,1,2\}$ in a natural way
corresponds to the function $h:V(K)\to\{0,1,2\}$, where
$h'(x)=h(x)$ for any $x\in V_1$, while $h'(e)$ where $e\in E(K')$
coincides with $h(y)$, where $y\in V(K)$ is the subdivision vertex
of $e$. In addition, the following conditions imposed to $h'$: \\
for every vertex $x\in V(K')$, $$\sum_{xy\in E(K')}{h'(xy)}\ge 2$$
and for every edge $xy\in E(K')$, $$h'(x)+h'(y)\ge 2,$$ are
equivalent to the corresponding function $h$ being a total
$\{2\}$-dominating function of $K$. Moreover, as
$f(K)=f_3(K)=|V_1\cap K|=|V(K')|$, it suffices to prove that there
exists a function $h'$ satisfying the above conditions such that
the total sum of values of $h'$ is at most $2|V(K')|$, to
establish the desired bound of the theorem in the component $K$.
Let $h'_1: V(K') \to\{0,1,2\}$ and $h'_2: E(K') \to\{0,1,2\}$ be
minimum weight $2$-vertex and $2$-edge covers of $K'$, respectively.
Define $h'(v) = h'_1(v)$ for each $v \in V(K')$ and $h'(e) =
h'_2(e)$ for each $e \in E(K')$. By definition of $2$-vertex cover
and $2$-edge cover, $h'$ satisfies the desired properties. Moreover,
by Theorem~\ref{lem:rho+tau}, the total sum of values of $h'$ is
exactly $2|V(K')|$. By the observation above, the corresponding
function $h:V(K)\to\{0,1,2\}$ is a total $\{2\}$-dominating
function of $K$ of weight at most $2|V(K)\cap V_1|$.

Clearly, in the same way as $K$ all connected components of $H''$
of order at least two can be analyzed, and so the function $h$ is
extended to all vertices of $H''$. Now, $H'$ is obtained from
$H''$ by adding connected components with only one vertex, and
they have not yet been considered in the proof. Recall that each
such vertex $x\in V(H')$ is in $V_1$, and we set $h(x)=0$ and
$h(y)=2$ to any neighbor $y\in V(H')$, with which the value of
$h$, which corresponds to $x$, is $2\cdot f(x)$, as desired.
Altogether we deduce that $h$ is a total $\{2\}$-dominating
function of $H'$ of weight at most $2|V(H')\cap V_1|$. As $H'$ is a
spanning subgraph of $H$, clearly $h$ is also a total
$\{2\}$-dominating function of $H$ of the same weight. Finally, we
set $g(x)=f_3(x)$ for all vertices in $V(G)\setminus V(H)$ and
$g(x)=h(x)$ for all vertices in $V(H)$. We conclude the main part
of this proof by observing that $g$ is a total
$\{2\}$-dominating function of $G$, with $g(V)\le 2f(V)$.
\end{proof}

\subsubsection*{Bound (9,3)}

\begin{prp}\label{prp:rgamma-w2-gamma-w2}
For every graph $G\neq K_1$, it holds that $\rgamawtwo(G)\le
2\gamawtwo(G)-2$.
\end{prp}

\begin{proof}
Let $f:V(G)\to \{0,1,2\}$ be a minimum weak $2$-dominating
function of $G$, and let $V_i = \{v\in V(G)\mid f(v) = i\}\,,$ for
$i\in \{0,1,2\}$. We consider two cases.

\medskip

\noindent{\it Case 1. $V_2\neq \emptyset$.} Define a function
$g:V(G)\to{\cal P}(\{a,b\})$ as follows:
$$g(v) = \left\{
  \begin{array}{ll}
    \{a,b\}, & \hbox{if $v \in V_1\cup V_2$;} \\
        \emptyset, & \hbox{otherwise}
  \end{array}
\right.$$ Since $f$ is weak $2$-dominating function of $G$, every
vertex $v\in V(G)$ with $f(v) = 0$ is adjacent to either two
vertices of $f$-weight $1$, or to one vertex of $f$-weight $2$.
Consequently, every vertex $v\in V(G)$ with $g(v) =\emptyset$
satisfies $g_\cup(N(v)) = \{a,b\}$. Thus, $g$ is a rainbow weak
$2$-dominating function of $G$, and we have $\rgamawtwo(G)\le
g(V(G)) = 2(|V_1|+|V_2|)\le 2(|V_1|+2|V_2|)-2 = 2f(V(G))-2 =
2\gamawtwo(G)-2$.

\medskip

\noindent{\it Case 2. $V_2= \emptyset$.} In this case, every
vertex $v\in V(G)$ with $f(v) = 0$ is adjacent to two vertices of
$f$-weight $1$. If $V_1 = V(G)$, then the function $g:V(G)\to
{\cal P}(\{a,b\})$ assigning $\{a\}$ to every vertex is a rainbow
weak $2$-dominating function of $G$, yielding in this case
$\rgamawtwo(G)\le g(V(G)) = |V(G)|\le 2|V(G)|-2= 2\gamawtwo(G)-2$,
where the inequality holds since $G\neq K_1$. If $V_1\neq V(G)$,
then there is a vertex $w\in V(G)$ with $f(w) = 0$ and
consequently $|V_1|\ge 2$. Let $u,v\in V_1$ be two distinct
vertices. Define a function $g:V(G)\to{\cal P}(\{a,b\})$ as
follows:
$$g(w) = \left\{
  \begin{array}{ll}
    \{a\}, & \hbox{if $w = u$;} \\
    \{b\}, & \hbox{if $w = v$;} \\
    \{a,b\}, & \hbox{if $w \in V_1\setminus\{u,v\}$;} \\
        \emptyset, & \hbox{otherwise.}
  \end{array}
\right.$$ Since $f$ is weak $2$-dominating function of $G$, every
vertex $v\in V(G)$ with $f(v) = 0$ is adjacent to either two
vertices of $f$-weight $1$, or to one vertex of $f$-weight $2$. It
follows that every vertex $w\in V(G)$ with $g(w) =\emptyset$ is
adjacent either to a vertex with $g$-label $\{a,b\}$, or to
vertices $u$ and $v$; in either case we have $g_\cup(N(w)) =
\{a,b\}$. Thus, $g$ is a rainbow weak $2$-dominating function of
$G$, and we have $\rgamawtwo(G)\le g(V(G)) = 2+2(|V_1|-2) =
2|V_1|-2 = 2f(V(G))-2 = 2\gamawtwo(G)-2$.
\end{proof}

\subsubsection*{Bound (13,3)}

\begin{prp}\label{prp:gamma-R-gamma-w-2}
For every graph $G$, it holds that $\gamaR(G)\le 2\gamawtwo(G)-1$.
\end{prp}

\begin{proof}
Let $f:V(G)\to\{0,1,2\}$ be a minimum weight weak $2$-dominating
function of $G$.

Suppose first that there exists a vertex $u\in V(G)$ such that
$f(u) = 2$. Then, let \hbox{$g:V(G)\to\{0,1,2\}$} be defined as
follows:
$$g(v)= \left\{
         \begin{array}{ll}
           2, & \hbox{if $f(v) \in \{1,2\}$;} \\
           0, & \hbox{otherwise.}
         \end{array}
       \right.$$
Note that $g$ is a Roman dominating function of $G$ of total
weight at most $2f(V(G))-2 = 2\gamawtwo(G)-2$. On the other hand,
if $f(v) \in \{0,1\}$ for all $v\in V(G)$, then let $u\in V(G)$ be
a vertex such that $f(u) = 1$, and let $g:V(G)\to\{0,1,2\}$ be
defined as follows:
$$g(v)= \left\{
         \begin{array}{ll}
           2, & \hbox{if $v\neq u$ and $f(v) = 1$;} \\
           1, & \hbox{if $v= u$;} \\
           0, & \hbox{otherwise.}
         \end{array}
       \right.$$
Note that $g$ is a Roman dominating function of $G$ of total
weight exactly $2f(V(G))-1 = 2\gamawtwo(G)-1$. Hence, we have the
desired inequality in each case.
\end{proof}

\subsubsection*{Bound (1,4)}

\begin{prp}\label{prp:gamma-gamma-set2}
For every graph $G$, $\gama(G)\le \gamasettwo(G)-1$.
\end{prp}

\begin{proof}
Let $f:V(G)\to\{0,1,2\}$ be a minimum $\{2\}$-dominating
function of $G$, and let $V_i = \{v\in V(G)\mid f(v) = i\}\,,$ for
$i\in \{0,1,2\}$. If $V_1 = \emptyset$, then $V_2$ is a
dominating set of $G$ and hence in this case $\gama(G) \le
|V_2|\le 2|V_2|-1= f(V(G))-1 = \gamasettwo(G)-1$. So we may
assume that $V_1\neq\emptyset$. Let $v_1\in V_1$ and consider the
set $D = (V_1\cup V_2)\setminus \{v_1\}$. Then, $|D| =|V_1|+
|V_2|-1\le |V_1|+2|V_2|-1 = f(V(G))-1 = \gamasettwo(G)-1$. Hence,
to show that $\gama(G)\le \gamasettwo(G)-1$, it suffices to
argue that $D$ is a dominating set of $G$, that is, that
every $v\in V(G)\setminus D$ has a neighbor in $D$.
If $v\in V(G)\setminus D$ and $v\neq v_1$, then $v$ has a neighbor from
$V_2$ or two neighbors from $V_1$, because $f$ is $\{2\}$-dominating
function of $G$. In the first case, this neighbor is clearly from $D$,
while in the second case, there is at least one neighbor of $v$ from $V_1$,
which is not equal to $v_1$. Thus $v$ has a neighbor from $D$.
It remains to show that $v_1$ has a neighbor from $D$, which is
also easy because $v_1\in V_1$, hence it must have a neighbor
from $V_1\cup V_2$, which is thus from $D$.
\end{proof}

\subsubsection*{Bound (2,5)}

\begin{prp}\label{prp:gamma-t-gamma-tset2}
For every graph $G$ without isolated vertices, $\gamat(G)\le
\gamatsettwo(G)-1$.
\end{prp}

\begin{proof}
Let $f:V(G)\to\{0,1,2\}$ be a minimum total $\{2\}$-dominating
function of $G$, and let $V_i = \{v\in V(G)\mid f(v) = i\}\,,$ for
$i\in \{0,1,2\}$. If $V_1 = \emptyset$, then $V_2$ is a total
dominating set of $G$ and hence in this case $\gamat(G) \le
|V_2|\le 2|V_2|-1= f(V(G))-1 = \gamatsettwo(G)-1$. So we may
assume that $V_1\neq\emptyset$. Let $v_1\in V_1$ and consider the
set $D = (V_1\cup V_2)\setminus \{v_1\}$. Then, $|D| =|V_1|+
|V_2|-1\le |V_1|+2|V_2|-1 = f(V(G))-1 = \gamatsettwo(G)-1$. Hence,
to show that $\gamat(G)\le \gamatsettwo(G)-1$, it suffices to
argue that $D$ is a total dominating set of $G$, that is, that
every $v\in V(G)$ has a neighbor in $D$. Since $f$ is a total
$\{2\}$-dominating function of $G$, every vertex of $v$ has either
two neighbors in $V_1$ or one neighbor in $V_2$. In particular,
every vertex of $v$ has  a neighbor in either $V_1\setminus
\{v_1\}$, or in $V_2$, and hence in $D$.
\end{proof}

\subsubsection*{Bound (7,6)}

\begin{prp}\label{prp:gamma-x2-gamma-2}
For every graph $G$ without isolated vertices, it holds that
${\displaystyle\gamma}\,\!\!_{{\scriptstyle \times} \! 2}(G)\le 2\gamatwo(G)-1$.
\end{prp}
\begin{proof}
Let $f:V\to\{0,1\}$ be a minimum $2$-dominating function of $G$, and
let $D=\{u\,|\,f(u)=1\}$. For all $x\in V\setminus D$, we have
$f(N[x])\ge 2$, hence the condition imposed on a double dominating
function is already fulfilled for these vertices. Let $y,w\in D$
be two neighbors of $x$ with $f$-value positive. Let $D'$ be a
superset of $D$ obtained from $D$ by adding to it vertex $x$
and for each vertex $z$ from $D\setminus \{y,w\}$,
adding an arbitrary vertex $u\in N(z)$.
Clearly, for any $z\in D$, we have $|N[z]\cap D'|\ge 2$, where one
of the vertices from $D'\cap N[z]$ is itself. Altogether we derive
that $D'$ is a double dominating set, that is, the function
$f':V\to\{0,1\}$, which sets value 1 precisely to the vertices
from $D'$, is a double dominating function of $G$, and its weight
is $|D'|\le 2|D|-1= 2\gamatwo(G)-1$.
\end{proof}

\subsubsection*{Bound (8,6)}

\begin{prp}\label{prp:gamma-tx2-gamma-2}
For every graph $G$ with $\delta(G) \ge 2$, $\gamatxtwo(G)\le
3\gamatwo(G)-2$.
\end{prp}

\begin{proof}
Let $f:V\to\{0,1\}$ such that for all $v\in V$ it holds that $f(v)
= 0~\Longrightarrow~f(N(v)) \ge 2$ and such that
$|D|=\gamatwo(G)$, where $D = \{v \in V \mid f(v) = 1\}$. So,
every vertex in $V \setminus D$ has at least two neighbors in $D$.
Let $D_2$ be the set of vertices in $D$ having at least two
neighbors in $D$, let $D_1$ be the set of vertices in $D$ having
exactly one neighbor in $D$, and let $D_0$ be the set of vertices
in $D$ having no neighbors in $D$. Since $\delta(G) \ge 2$, we can
always define a set $D'$ of size at most $3|D_0|+2|D_1|+|D_2|$ by
adding to $D$ one neighbor of $v$ in $V\setminus D$ for each
vertex $v$ in $D_1$, and two neighbors of $v$ in $V\setminus D$
for each vertex $v$ in $D_0$. If $D_2\neq \emptyset$, then $|D_2
\cup D_1| \geq 3$. Thus $|D'| \le 3|D_0|+2|D_1|+|D_2| \le
3(|D_0|+|D_1|+|D_2|) - |D_1| - 2|D_2| \le 3|D|-2$. If $D_2=
\emptyset$ but $D_1\neq \emptyset$, then $|D_1| \geq 2$. Again,
$|D'| \le 3|D|-2$. In both cases, we can define $f':V\to\{0,1\}$
such that, for every vertex $v\in V$, $f'(v) = 1$ if and only if
$v \in D'$. Function $f'$ is a total double dominating function,
thus $\gamatxtwo(G)\le 3\gamatwo(G)-2$.

Let now suppose $D_1=D_2=\emptyset$, so $D=D_0$ is an independent
set, and consider the bipartite graph $G'$ obtained from $G$ by
deleting all edges with both endpoints in $V\setminus D$. Since
$f$ is a $2$-dominating function, $\delta(G') \ge 2$, and $G'$
contains an even cycle $C$ of length at least four. We can
therefore define a set $D'$ by adding to $D$ the vertices of $C$
and two neighbors of $v$ in $V\setminus D$ for each vertex $v$ in
$D \setminus C$. Function $f':V\to\{0,1\}$ such that, for every
vertex $v\in V$, $f'(v) = 1$ if and only if $v \in D'$ is a total
double dominating function, and as $|D'| \le 3|D|-2$, we obtain
$\gamatxtwo(G)\le 3\gamatwo(G)-2$.
\end{proof}

\subsubsection*{Bound (8,7)}

\begin{prp}\label{prp:gamatxtwo-gamaxtwo}
For every graph $G$ with $\delta(G)\ge 2$, we have
$\gamatxtwo(G)\le 2\gamaxtwo(G)-1$.
\end{prp}

\begin{proof}
Let $D\subseteq V(G)$ be a minimum double dominating set of $G$.
Then, every vertex in $D$ has at least one neighbor in $D$, and
every vertex in $V(G)\setminus D$ has at least two neighbors in
$D$. Let $D_1 = \{v\in D\mid d_{G[D]}(v) = 1\}$. If $D_1 =
\emptyset$, then $D$ is also a total double dominating set, in
which case $\gamatxtwo(G)\le |D| = \gamaxtwo(G)\le
2\gamaxtwo(G)-1$.

Now let $D_1 \neq \emptyset$. Fix a vertex $v\in D_1$, and let
$w\in N(v)\setminus D$. Note that such a vertex exists since $v\in
D_1$ and $d_G(v)\ge 2$. Since $w\not\in D$, there exists a
neighbor of $w$ in $D$, say $v'$, such that $v'\neq v$. For each
vertex $x\in D_1\setminus\{v,v'\}$, let $x'$ denote an arbitrary
neighbor of $x$ outside $D$, and set $D' = D\cup \{w\}\cup X$,
where $X = \{x'\mid x\in D_1\setminus\{v,v'\}\}$.

We claim that $D'$ is a total double dominating set of $G$. To see
this, consider an arbitrary vertex $x\in V(G)$; we need to check
that $x$ has at least two neighbors in $D'$. If $x \in \{v,v'\}$,
then $|N(x)\cap D|\ge 1$ and $N(x)\cap D' \supseteq (N(x)\cap
D)\cup \{w\}$ (disjoint union). If $x \in D_1\setminus\{v,v'\}$,
then again $|N(x)\cap D|\ge 1$ and $N(x)\cap D' \supseteq
(N(x)\cap D)\cup \{x'\}$ (disjoint union). If $x \in V(G)\setminus
D_1$, then $|N(x)\cap D|\ge 2$ and $N(x)\cap D' \supseteq N(x)\cap
D$. In either case,  the conclusion follows.

The above implies that $\gamatxtwo(G)\le |D'|$, hence it suffices
to show that $|D'|\le 2|D|-1 = 2\gamaxtwo(G)-1$. If $v'\in D_1$,
then $|X|\le |D_1|-2\le |D|-2$, hence $|D'| \le |D|+1+|X|\le
2|D|-1$. If $v'\not\in D_1$, then $v'\in D\setminus D_1$, hence
$|D_1|\le |D|-1$ and again we have $|X|\le |D_1|-1\le |D|-2$, thus
the same argument applies.
\end{proof}

\subsubsection*{Bound (11,10)}

\begin{prp}\label{prp:rgamma-x2-rgamma-2}
For every graph $G$ with no isolated vertices, $\rgamaxtwo(G)\le
2\rgamatwo(G)$.
\end{prp}

\begin{sloppypar}
\begin{proof}
Let $f:V\to\{\emptyset,\{a\},\{b\}\}$ such that for all $v\in V$
it holds that $f(v) = \emptyset~\Longrightarrow~|f_\cup(N(v))| \ge
2$ and such that the total weight $f(V)=\rgamatwo(G)$. Let $D_x =
\{v \in V \mid f(v) = \{x\}\}$, for $x \in \{a,b\}$, and $D_0 =
\{v \in V \mid f(v) = \emptyset\}$. Notice that
$\rgamatwo(G)=|D_a|+|D_b|$. Also, every vertex $v$ in $D_0$ has a
neighbor in $D_a$ and a neighbor in $D_b$, in particular it
satisfies $|f_\cup(N[v])| \ge 2$.

Vertices $v$ of $D_a$ having neighbors in $D_b$ and vertices $v$
of $D_b$ having neighbors in $D_a$ also satisfy $|f_\cup(N[v])|
\ge 2$.

We will do the following process, successively, for all the
vertices in $D_a$ having neighbors only in $D_a \cup D_0$. Let $v$
be such a vertex. If $v$ has a neighbor $w$ in $D_0$, then we
update $f(w) := \{b\}$, and update accordingly the sets $D_b$ and
$D_0$. Now, $|f_\cup(N[v])| \ge 2$, and still $|f_\cup(N[z])| \ge
2$ for every $z$ satisfying that before the update. If $v$ has no
neighbor in $D_0$, then it has at least one neighbor in $D_a$. We
update $f(v) := \{b\}$, and update accordingly the sets $D_a$ and
$D_b$. Again, $|f_\cup(N[v])| \ge 2$, and still $|f_\cup(N[z])|
\ge 2$ for every $z$ satisfying that before the update, because
$v$ had no neighbors in $D_b \cup D_0$. If, with the new
definitions of $D_a, D_b, D_0$, there are still vertices in $D_a$
having neighbors only in $D_a \cup D_0$, we repeat the process.
Notice that the size of this set strictly decreases on each step,
while the size of vertices in $D_b$ having neighbors only in $D_b
\cup D_0$ never increases.

Once the set of vertices in $D_a$ having neighbors only in $D_a
\cup D_0$ is empty, we start processing vertices in $D_b$ having
neighbors only in $D_b \cup D_0$ analogously, interchanging the
roles of $a$ and $b$. Notice that the size of the set of vertices
in $D_b$ having neighbors only in $D_b \cup D_0$ strictly
decreases on each step, while we never create vertices in $D_a$
having neighbors only in $D_a \cup D_0$. So, once the former set
is empty, all the vertices $z$ in $V$ satisfy $|f_\cup(N[z])| \ge
2$, hence $f$ is a rainbow double dominating function of $G$.
Notice that we have done at most $|D_a|+|D_b|$ steps and on each
step we give a nonempty label to at most one vertex in $D_0$. So
the new weight $f(V)$ is at most $2\rgamatwo(G)$.
\end{proof}
\end{sloppypar}

\subsubsection*{Bound (1,11)}

\begin{prp}\label{prp:gamma-rgamma-x2}
For every graph $G$ without isolated vertices, $\gama(G)\le
\frac{1}{2}\rgamaxtwo(G)$.
\end{prp}

\begin{proof}
Let $G$ be a graph without isolated vertices, and let
$f:V(G)\to\{\emptyset,\{a\},\{b\}\}$ be a minimum weight rainbow
double dominating function of $G$. Consider the sets $A = \{v\in
V(G)\mid f(v) = \{a\}\}$ and $B = \{v\in V(G)\mid f(v) = \{b\}\}$.
Without loss of generality, we may assume that $|A|\le |B|$. Note
that every vertex in $V(G)\setminus A$ has a neighbor in $A$, thus $A$ is a
dominating set of $G$, implying $\gama(G)\le |A|\le
\frac{1}{2}f(V(G)) = \frac{1}{2}\rgamaxtwo(G)$.
\end{proof}

\subsubsection*{Bound (2,12)}

\begin{prp}\label{prp:gamma-t-rgamma-tx2}
For every graph $G$ of total domatic number at least $2$,
$\gamat(G)\le \frac{1}{2}\rgamatxtwo(G)$.
\end{prp}

\begin{proof}
Let $G$ be a graph with total domatic number at least~$2$, and let
$f:V(G)\to\{\emptyset,\{a\},\{b\}\}$ be a minimum weight rainbow
total double dominating function of $G$. Consider the sets $A =
\{v\in V(G)\mid f(v) = \{a\}\}$ and $B = \{v\in V(G)\mid f(v) =
\{b\}\}$. Without loss of generality, we may assume that $|A|\le
|B|$. Note that every vertex in $G$ has a neighbor in $A$. In
particular, $A$ is a total dominating set of $G$, implying
$\gamat(G)\le |A|\le \frac{1}{2}f(V(G)) =\frac{1}{2}\rgamatxtwo.$
\end{proof}

\subsubsection*{Bound (1,13)}

\begin{prp}\label{prp:gamma-gammaR}
For every graph $G$ with at least one edge, $\gama(G)\le
\gamaR(G)-1$.
\end{prp}
\begin{proof}
Let $f:V(G)\to\{0,1,2\}$ be a minimum weight Roman dominating
function of $G$, and let $V_i = \{v\in V(G)\mid f(v) = i\}\,$ for
$i\in \{0,1,2\}$. We may assume that in addition $f$ is chosen
in such a way that $V_2\neq \emptyset$. Indeed let $uv\in E(G)$,
and assume that $f(u)=1=f(v)$. But then changing $f(u)=1$ to $f(u)=2$
and $f(v)=1$ to $f(v)=0$ yields a Roman dominating function of
the same weight. Note that $V_1\cup V_2$ is a dominating set,
which implies $\gamma(G)\le |V_1|+|V_2|< |V_1|+2|V_2|=\gamaR(G)$,
hence $\gama(G)\le \gamaR(G)-1$.
\end{proof}

\section{Examples of sharpness and non-existence}
\label{sec:sharp}

\subsection{Main families of graphs used in the proofs}

In this section we present several families of graphs
that will be used in the proofs of sharpness of the bounds
from Table~\ref{table-bounds} or in the proofs of non-existence of such
bounds indicated in the same table. Here we present
the families that are used in several instances, while
those that are used just once or twice will be presented
along with Table~\ref{table-sharp}, which summarizes the families
used for each proof.

Given a graph $G$, by $kG$ we define the disjoint union of $k$ copies
of $G$. Hence $kK_2$, resp. $kC_4$ stands for the graphs on $k$ components,
each component being the connected graph on $2$ vertices, resp.~the {\em square}, that is, the cycle on four vertices}.

Next, for a graph $G$, $S(G)$ denotes its subdivision graph, obtained
from $G$ by subdividing each of its edges exactly once.
In $S(G)$ we distinguish between {\em original} and {\em subdivided} vertices
that correspond to the vertices of $G$ and to those internal vertices of the paths on three vertices replacing the edges of $G$
to obtain $S(G)$ respectively.
The stars with $n$ leaves are denoted by $K_{1,n}$. We denote by $S(K_{1,n})^-$ the graph obtained from the subdivision graph
$S(K_{1,n})$ of the $n$-star by deleting a leaf.

The graph $H$ is the tree on 6 vertices, in which each of the two adjacent non-leaves is adjacent to two leaves.

We denote by $F_n^3$ the graph on $2n+1$ vertices, obtained from $nK_3$ by identifying (gluing) one vertex
of each triangle to a single vertex. Similarly, $F_n^4$ is the graph on $3n+1$ vertices, obtained from $nC_4$
by identifying (gluing) one vertex of each square to a single vertex. See Fig.~\ref{fig:Fn}.

\begin{figure}[h!]
\begin{center}
\includegraphics{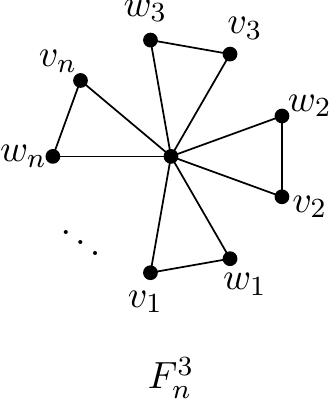}\qquad
\includegraphics{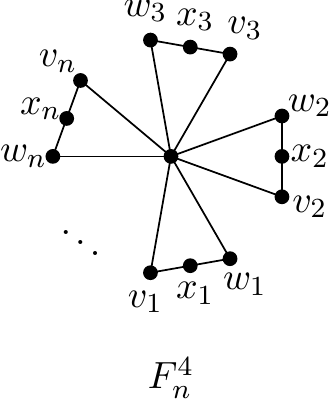}
\end{center}
\caption{Graphs $F_n^3$ and $F_n^4$}\label{fig:Fn}
\end{figure}

For $n\ge 2$, let $K_n^{**}$ denote the graph obtained from the complete graph of order $n$ by gluing two
new triangles along each edge; see Fig.~\ref{fig:Kn**andSKn2} for an example.
In other words, for each pair $x,y$ of vertices in the complete graph $K_n$ two vertices are added, each of which is adjacent
only to $x$ and $y$. The added vertices in $K_n^{**}$ (that is, those of degree $2$) will be called {\em triangle} vertices; vertices that are not triangle will be called {\em original}.

\begin{figure}[h!]
\begin{center}
\includegraphics[width=80mm]{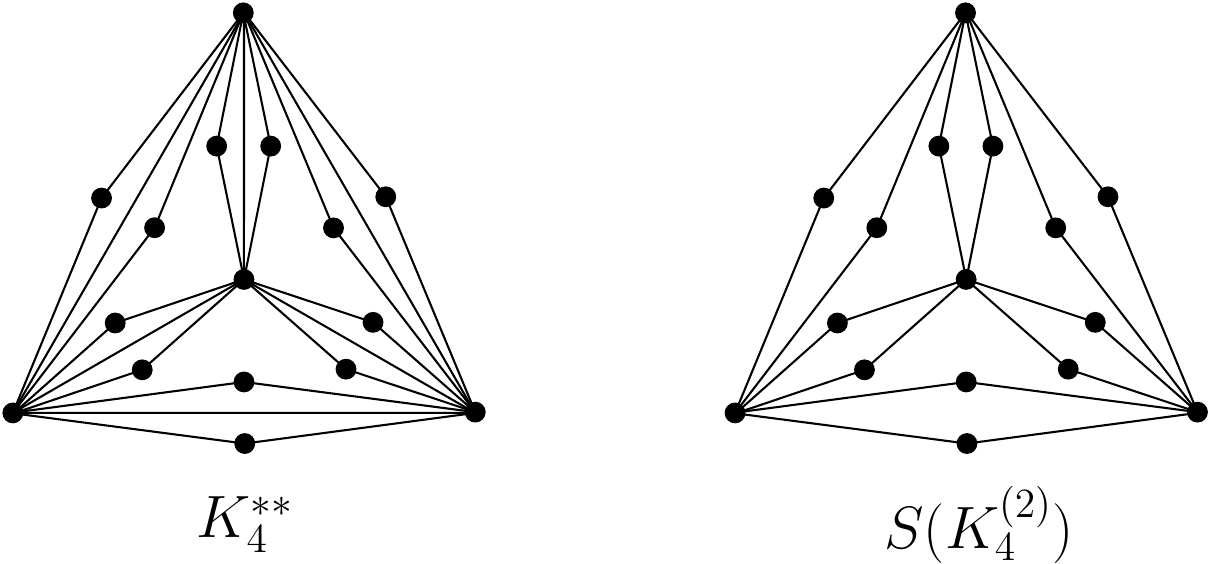}
\end{center}
\caption{Examples of graphs $K_n^{**}$ and $S(K_n^{(2)})$}\label{fig:Kn**andSKn2}
\end{figure}

By $G^{(k)}$ we denote the multigraph obtained from a graph $G$ by replacing each edge with $k$ parallel edges.
In particular, $K_n^{(2)}$ is the multigraph obtained from the complete graph of order $n$ by duplicating each edge,
and so $S(K_n^{(2)})$ is its subdivision graph; see Fig.~\ref{fig:Kn**andSKn2} for an example. Note that $S(K_n^{(2)})$ is obtainable from $K_n^{**}$ by deleting all edges joining pairs of original vertices.
Next, the graph $S(K_3^{(n)})$ is the subdivision graph of the multigraph
$K_3^{(n)}$ (i.e., the subdivision graph of the multigraph obtained from $K_3$ by adding $n-1$ parallel edges
between each pair of vertices). See Fig.~\ref{fig:SK3nQn} for an example.

\begin{figure}[h!]
\begin{center}
\includegraphics[width=80mm]{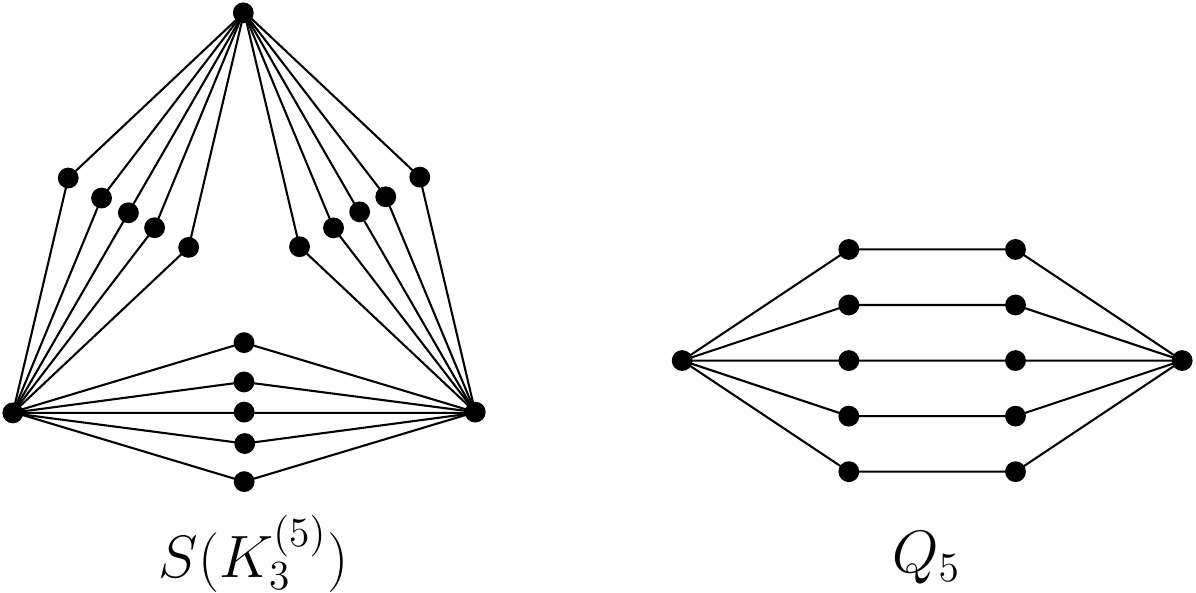}
\end{center}
\caption{Examples of graphs $S(K_3^{(n)})$ and $Q_n$}\label{fig:SK3nQn}
\end{figure}

We will denote by $Q_n$ the graph that can obtained from the multigraph $K_2^{(n)}$ (two vertices, connected by $n$ parallel edges)
by subdividing each edge twice, i.e., each edge is replaced by the path $P_4$ (the so-called double subdivision graphs of $K_2^{(n)}$).
See Fig.~\ref{fig:SK3nQn} for an example.

\begin{sloppypar}
Finally, graphs $T_n$ are defined as follows. Let $V(T_n)=\{v_1, \dots, v_n,w_1, \dots, w_n,$
$s_1, s_2, s_3, t_1, \dots, t_5\}$, so that
$s_1,s_2,s_3$ induce a triangle, $t_1, \dots, t_5$ induce a $C_5$,
$s_1$ and $s_3$ are adjacent to $v_i$ for every $i\in\{1,\ldots,n\}$,
$t_1$ and $t_5$ are adjacent to $w_i$ for every $i\in\{1,\ldots,n\}$,
$v_iw_i \in E(T_n)$ for every $i\in\{1,\ldots,n\}$, and there are no other
edges; see Fig.~\ref{fig:Tn}.
\end{sloppypar}

\begin{figure}[h!]
\begin{center}
\includegraphics[width=0.39\linewidth]{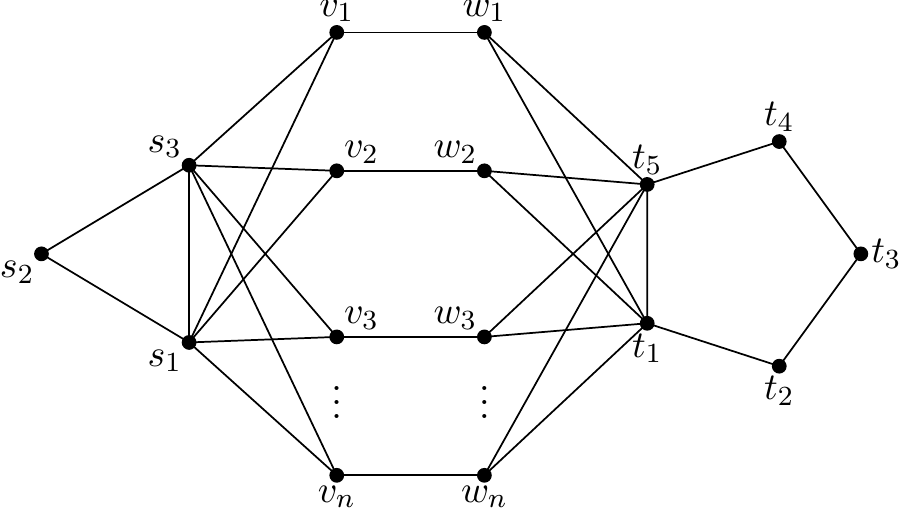}
\end{center}
\caption{Graphs $T_n$}\label{fig:Tn}
\end{figure}

It is for instance easy to see that for each $n$, the set $\{s_1,t_1,t_3\}$ is a minimum
dominating set of $T_n$. We will discuss the values of several other parameters of $T_n$ and the consequences
for Table~\ref{table-families} in Subsection~\ref{ss:unbound}.

The sharpness of the bounds in Table~\ref{table-bounds} will be demonstrated using the values of relevant parameters
on graphs families summarized in rows A--G of Table~\ref{table-families}.
In addition, we will use three more families, not described in this table (each of them used only for three invariants).
Furthermore, the results for families in rows X, Y, Z, W will be used in proving that there is no bound between certain
pairs of invariants.

\newcommand{\vacio}{}

\begin{sidewaystable}
\centering { \small
\renewcommand{\arraystretch}{1.4}
\tabcolsep=0.125cm
\begin{tabular}{|c|c||c|c|c|c|c|c|c|c|c|c|c|c|c|c|}
  \hline
\multicolumn{2}{|c||}{\multirow{2}{*}{$H$ vs. $\rho(H)$ }} & 1 & 2 & 3 & 4 & 5 & 6 & 7 & 8 & 9 & 10 & 11 & 12 & 13\\
\hhline{~~-------------}
\multicolumn{2}{|c||}{} &  $\gamma$ & $\gamat$ & $\gamawtwo$ & $\gamasettwo$ & $\gamatsettwo$ & $\gamatwo$ & $\gamaxtwo$ & $\gamatxtwo$   & $\rgamawtwo$ & $\rgamatwo$ & $\rgamaxtwo$ & $\rgamatxtwo$ &  $\gamaR$   \\
\hline\hline


  A & $kK_2$
 & $k$ & $2k$ & $2k$   & $2k$ & $4k$ & $2k$ &
 $2k$ & $\vacio$ &
  $2k$ & $2k$ & $2k$ & $\vacio$ & $2k$  \\
\hline

  B & $kC_4$
 & $2k$ & $2k$ & $2k$ & $\vacio$ &  $4k$  & $2k$ & $\vacio$ & $4k$
  & $2k$ & $2k$ & $4k$ & $4k$ & $3k$\\
\hline

  C & $K_n^{**},n\ge 3$

&  $n-1$ & $n-1$ & $n$ & $n$ & $n$ & $n$ & $n$ & $n$
  & $2n-2$ & $\vacio$ & $\vacio$ & $\vacio$ & $2n-2$  \\

\hline

  %E & $S(K_{1,n})^-$

% &  $n$ & $\vacio$ & $\vacio$ & $2n$ & $\vacio$ & $\vacio$ &
 %  $\vacio$ & $\vacio$
 % & $\vacio$ &  $\vacio$ & $\vacio$ & $\vacio$ & $n+1$ \\
% \hline

  D & $kH$

 & $\vacio$ & $2k$ & $4k$ & $4k$ & $4k$ & $\vacio$ & $\vacio$  & $\vacio$
  & $4k$ & $\vacio$   & $\vacio$ & $\vacio$ & $4k$ \\
\hline

  E & $kK_{4,4}$

 & $\vacio$ & $\vacio$ & $4k$ & $4k$ & $4k$ & $4k$ & $4k$  & $4k$
  & $4k$ & $4k$   & $4k$ & $4k$ & $4k$ \\
\hline

  F & $F_n^3$
 & $\vacio$ & $\vacio$ &$ \vacio$ & $\vacio$
  & $\vacio$ & $\vacio$ & $n+1$ &    $2n+1$
  & $\vacio$ &  $\vacio$  & $n+1$ & $\vacio$ & $\vacio$\\
 \hline

  G & $F_n^4$
 & $\vacio$ & $\vacio$ & $n+1$ & $2n+1$
  & $\vacio$ & $n+1$ & $2n+1$ &    $3n+1$
  & $n+1$ &  $n+1$  & $\vacio$ & $\vacio$ & $\vacio$\\
 \hline
\hline

  X & $K_{1,n}$,$n\ge 2$
 & $1$ & $2$ & $2$ & $2$
  & $4$ & $n$ & $n+1$ &    $\vacio$
  & $2$ &  $n$  & $n+1$ & $\vacio$ & $2$\\
 \hline

Y & $S(K_3^{(n)})$, $n\ge 3$ & $\vacio$ & $\vacio$  & $\vacio$  & $\vacio$ &
$\vacio$ & $3$ & $5$ & $6$   & $\vacio$ &
   $n+3$  & $n+4$ & $\vacio$ & $\vacio$ \\
\hline

Z & $Q_n$, $n\ge 3$ & $2$ & $4$  & $4$  & $4$ &
$8$ & $\vacio$ & $\vacio$ & $2n+2$   & $4$ &
   $\vacio$  & $\vacio$ & $\vacio$ & $4$ \\
\hline

W & $T_n$ & $3$ & $5$  & $5$  & $6$ &
$8$ & $5$ & $6$ & $8$   & $6$ &
   $6$  & $6$ & $n+9$ & $6$ \\
\hline

\end{tabular}
}\caption{Summary of proofs for the entries of Table
\ref{table-sharp}. Only the entries that are used in the proof of Table
\ref{table-sharp} are non-empty. Each nonempty entry of the table
gives the value of the domination parameter indexing the corresponding column on
the graph indexing the row. The entries show that the bounds are sharp, or that
a certain parameter cannot be bounded from above by a function of another one.}
\label{table-families}
\end{sidewaystable}

Since most of the values in Table~\ref{table-families} can be proved as an easy exercise, we
will only present here a proof for those that are a bit more involved.

\begin{claim}
For $n\ge 3$, we have
$\rho(K_n^{**})= n$ for every
$\rho\in \{\gamawtwo,\gamasettwo,\gamatsettwo,\gamatwo,\gamaxtwo,\gamatxtwo\}$, and
$\rgamawtwo(K_n^{**}) = \gamaR(K_n^{**})=2n-2$.
\end{claim}

\begin{proof}
Recall that $K_n^{**}$ is the graph obtained from the complete graph of order $n$
by gluing two triangles along each edge, and let us denote $K_n^{**}$ shortly by $G_n$.
Let $V_n$ be the set of original vertices in $K_n^{**}$ (the vertices of
the complete subgraph $K_n$), and $V_n^{**}$ the set $V(G_n)\setminus V_n$
(triangle vertices).
On the one hand, it is easy to see that $\gamatxtwo(G_n)\le n$, since assigning $1$ to
each vertex of $V_n$ and $0$ to each other vertex yields a total double dominating function of $G_n$ of
total weight $n$. On the other hand, we will now show that
$\rgamawtwo(G_n)\ge 2n-2$. Since the bounds from Table~\ref{table-bounds}
imply that $\rgamawtwo(G_n)\le \gamaR(G_n)\le 2\gamatxtwo(G_n)-2$
and $\rho(G_n)\le \gamatxtwo(G_n)$ for all
$\rho \in \{\gamawtwo,\gamasettwo,\gamatsettwo,\gamatwo,\gamaxtwo\}$, the claim will then follow.

Suppose for a contradiction that there is a minimum rainbow weak $2$-dominating function
$f:V(G_n)\to{\cal P}(\{a,b\})$ of $G_n$ of total weight at most
$2n-3$. We first argue that we may assume without loss of
generality that $f(v) = \emptyset$ for all vertices $v\in
V_n^{**}$. Indeed, if $f(v) \neq \emptyset$ for some vertex $v\in
V_n^{**}$, then the minimality of $f$ implies that
 $f(v') \neq \emptyset$, where $v'$ denotes the unique vertex
 with $v'\neq v$ and $N(v) = N(v')$.
Hence, assigning $\{a,b\}$ to one of the neighbors of $v$ and
assigning  $\emptyset$ to each of $v$ and $v'$ maintains
feasibility without increasing the total weight. Performing the
above procedure as long as necessary eventually results in a
function $f$ such that $f(v) = \emptyset$ for all vertices $v\in
V_n^{**}$. Let $V_a = \{v\in V_n\mid f(v) = \{a\}\}$; the sets
$V_b$ and $V_{ab}$ are defined similarly. Since $2|V_{ab}|\le
f(V(G_n))\le 2n-3$, we infer that $|V_{ab}|\le n-2$. The fact that
$f$ is a rainbow weak $2$-dominating function of $G_n$ with $f(v)
= \emptyset$ for each $v\in V_n^{**}$, implies that $|V_a|\le 1$
and $|V_b|\le 1$. If either $V_a=\emptyset$ or $V_b=\emptyset$,
then there exists a vertex $v\in V_n^{**}$ such that $|f(N(v))|\leq 1$, contrary to the fact that
$f$ is a rainbow weak $2$-dominating function and $f(v) = \emptyset$.
Consequently, $|V_a| = |V_b| = 1$, which implies $|V_{ab}| = n-2$. But now,
$f(V(G_n))= 2n-2$, a contradiction.
\end{proof}

As most other values in rows A-G in Table~\ref{table-families} are straightforward
(in particular the values for $kK_2,kC_4,kH,kK_{4,4},F_n^3$),
we continue with the class $F_n^4$; these are the graphs obtained from a set of $n$ cycles $C_4$
by identifying a vertex from each of the 4-cycles to a single vertex.
Let us denote by $v$ the unique vertex of degree $2n$ in $F_n^4$. First, note that
$\rgamawtwo(F_n^4)\le \rgamatwo(F_n^4)\le n+1$, which is proven by
the inequality $\rgamawtwo(F_n^4)\le \rgamatwo(F_n^4)$ (see Table \ref{table-bounds}) and the function
$f:V(F_n^4)\to {\cal P}(\{a,b\})$, which assigns $f(v)=\{a\}$, $f(u)=\{b\}$ to all non-neighbors
$u$ of $v$, and $f(x)=\emptyset$  to the remaining vertices.
The values $\gamawtwo(F_n^4)=\gamatwo(F_n^4)=\rgamawtwo(F_n^4)=\rgamatwo(F_n^4)=n+1$
can be derived from the following result and the corresponding upper bounds in row $4$ in Table~\ref{table-bounds}.

\begin{claim}
\label{cl:F4}
For $n\ge 3$, $\gamasettwo(F_n^4)=2n+1$.
\end{claim}
\begin{proof}
Let $S$ denote the set of all vertices at distance two from the central vertex $v$ of $F_n^4$.
Let us denote $F_n^4$ shortly by $G_n$.
Let us show that $\gamasettwo(G_n)\ge 2n+1$. Indeed, suppose for a contradiction that there exists
a $\{2\}$-dominating function $f:V(G_n)\to\{0,1,2\}$ of $G_n$ of total weight at most $2n$.
Since all vertices in $S$ have pairwise disjoint closed neighborhoods,
each of them needs weight $2$ to dominate vertices in $S$.
Since $v$ is not contained in any closed neighborhood of a vertex from $S$, we deduce
$f(v) = 0$; moreover, $f(N[u]) = 2$ for each vertex $u\in S$.
In order to dominate the neighbors of $v$,
we must have $f(u) = 2$ for all $u\in S$. But this implies that $f(N[v]) = 0$, a contradiction.
This shows that $\gamasettwo(G_n)\ge 2n+1$.
Since $\gamasettwo(G)\le 2\gamawtwo(G)-1\le 2(n+1)-1$, we derive that the claim is correct.
\end{proof}

Note that Claim~\ref{cl:F4} also implies that $\gamaxtwo(F_n^4)=2n+1$.
To see that $\gamatxtwo(F_n^4)\ge 3n+1$
one needs only to observe that for each vertex $u$ with degree 2
a total double dominating function $f$ of $G$ must assign $1$ to each
of the neighbors of $u$. On the other hand, assigning $1$ to all vertices
yields a total double dominating function of $F_n^4$, thus $\gamatxtwo(F_n^4)=3n+1$.

Some of the (not straightforward) values in rows X, Y, Z, W will be proven in Subsection~\ref{ss:unbound}
along with the proofs of unboundedness relations.

\begin{sidewaystable}
\centering { \small
\renewcommand{\arraystretch}{1.4}
\tabcolsep=0.125cm
\begin{tabular}{|c|c||c|c|c|c|c|c|c|c|c|c|c|c|c|c|}
  \hline
\multicolumn{2}{|c||}{\multirow{2}{*}{$\rho\le f(\rho')$}} & 1 & 2 & 3 & 4 & 5 & 6 & 7 & 8 & 9 & 10 & 11 & 12 & 13\\
\hhline{~~-------------}
\multicolumn{2}{|c||}{} &  $\gamma$ & $\gamat$ & $\gamawtwo$ & $\gamasettwo$ & $\gamatsettwo$ & $\gamatwo$ & $\gamaxtwo$ & $\gamatxtwo$   & $\rgamawtwo$ & $\rgamatwo$ & $\rgamaxtwo$ & $\rgamatxtwo$ &  $\gamaR$   \\
\hline\hline

  1 & $\gama$
& $=$ & $kC_4$  & $kC_4$ & $K_n^{**},n\ge 3$ & $K_n^{**},n\ge 3$ & $kC_4$ &
$K_n^{**},n\ge 3$ & $K_n^{**},n\ge 3$
  & $kC_4$ & $kC_4$ & $kK_2$ & $kC_4$ & $[S(K_{1,n})^-]$ \\
\hline

  2 & $\gamat$

& $kK_2$ & $=$ & $S(K_{2n+1})$ & $kK_2$ & $K_n^{**}, n\ge 3$ &
$S(K_{2n+1})$  & $kK_2$ & $K_n^{**},n\ge 3$
  & $kK_2$ & $kK_2$ & $kK_2$ & $kC_4$ & $kK_2$   \\
  \hline

  3 & $\gamawtwo$

&  $kK_2$ & $kH$ & $=$ & $kK_2$ & $kH$ & $kK_2$ & $kK_2$ &
$K_n^{**},n\ge 3$
  & $kK_2$ & $kK_2$ & $kK_2$ & $kK_{4,4}$ & $kK_2$  \\

\hline

  4 & $\gamasettwo$
 & $kK_2$ & $kH$ & [$F_n^4]$   & $=$ & $K_n^{**},n\ge 3$ & $[F_n^4$] &
 $kK_2$ & $K_n^{**},n\ge 3$ &
  [$F_n^4$] & [$F_n^4$] & $kK_2$ & $kK_{4,4}$ & $[S(K_{1,n})^-]$  \\
\hline

  5 & $\gamatsettwo$
 & $kK_2$ & $kK_2$ & $kK_2$ & $kK_2$ &  $=$  & $kK_2$ & $kK_2$ & $K_n^{**},n\ge 3$
  & $kK_2$ & $kK_2$ & $kK_2$ & $kC_4$ & $kK_2$\\
\hline

  6 & $\gamatwo$

 &  $\boldsymbol{K_{1,n}}$ & $\boldsymbol{K_{1,n}}$ & $\boldsymbol{K_{1,n}}$ & $\boldsymbol{K_{1,n}}$ & $\boldsymbol{K_{1,n}}$ & $=$ &
   $kK_2$ & $K_n^{**},n\ge 3$
  & $\boldsymbol{K_{1,n}}$ &  $kK_2$ & $kK_2$ & $kK_{4,4}$ & $\boldsymbol{K_{1,n}}$ \\
\hline

  7 & $\gamaxtwo$

 & $\boldsymbol{K_{1,n}}$ & $\boldsymbol{K_{1,n}}$ & $\boldsymbol{K_{1,n}}$ & $\boldsymbol{K_{1,n}}$ & $\boldsymbol{K_{1,n}}$ & $F_n^4$ & $=$  & $K_n^{**},n\ge 3$
  & $\boldsymbol{K_{1,n}}$ & $F_n^4$   & $kK_2$ & $kK_{4,4}$ & $\boldsymbol{K_{1,n}}$ \\
\hline

  8 & $\gamatxtwo$
 & $\boldsymbol{Q_n}$ & $\boldsymbol{Q_n}$ & $\boldsymbol{Q_n}$ & $\boldsymbol{Q_n}$
  & $\boldsymbol{Q_n}$ & $F_n^4$ & $F_n^3$ &    $=$
  & $\boldsymbol{Q_n}$ &  $F_n^4$  & $F_n^3$ & $kC_4$ & $\boldsymbol{Q_n}$\\

 \hline
9 & $\rgamawtwo$ & $kK_2$ & $kH$  & $K_n^{**},n\ge 3$  &
$K_n^{**},n\ge 3$ & $K_n^{**},n\ge 3$ & $K_n^{**},n\ge 3$ &
$K_n^{**},n\ge 3$ & $K_n^{**},n\ge 3$ & $=$ &
   $kK_2$  & $kK_2$ & $kK_{4,4}$ & $kK_2$ \\
\hline

  10 & $\rgamatwo$
 &  $\boldsymbol{K_{1,n}}$ & $\boldsymbol{K_{1,n}}$ & $\boldsymbol{K_{1,n}}$ & $\boldsymbol{K_{1,n}}$ & $\boldsymbol{K_{1,n}}$& $\boldsymbol{S(K_3^{(n)})}$ &
 $\boldsymbol{S(K_3^{(n)})}$   & $\boldsymbol{S(K_3^{(n)})}$  & $\boldsymbol{K_{1,n}}$&
 $=$  & $kK_2$ & $kK_{4,4}$ & $\boldsymbol{K_{1,n}}$ \\

\hline
  11 & $\rgamaxtwo$
 &  $\boldsymbol{K_{1,n}}$ & $\boldsymbol{K_{1,n}}$ & $\boldsymbol{K_{1,n}}$ & $\boldsymbol{K_{1,n}}$ & $\boldsymbol{K_{1,n}}$ & $\boldsymbol{S(K_3^{(n)})}$ & $\boldsymbol{S(K_3^{(n)})}$   & $\boldsymbol{S(K_3^{(n)})}$  & $\boldsymbol{K_{1,n}}$ & $kC_4$ & $=$   & $kC_4$ & $\boldsymbol{K_{1,n}}$ \\

\hline
  12 & $\rgamatxtwo$
  & $\boldsymbol{T_n}$ & $\boldsymbol{T_n}$ & $\boldsymbol{T_n}$ & $\boldsymbol{T_n}$  & $\boldsymbol{T_n}$
  & $\boldsymbol{T_n}$ &  $\boldsymbol{T_n}$ & $\boldsymbol{T_n}$
  & $\boldsymbol{T_n}$  &  $\boldsymbol{T_n}$ & $\boldsymbol{T_n}$ & $=$    & $\boldsymbol{T_n}$ \\

\hline
  13 & $\gamaR$
  & $kK_2$ & $kH$ & $S(K_n^{(2)})$ & $K_n^{**},n\ge 3$ & $K_n^{**},n\ge 3$ & $S(K_n^{(2)})$
   & $K_n^{**},n\ge 3$ & $K_n^{**},n\ge 3$ & $kC_4$ & $kC_4$ & $kK_2$ & $kK_{4,4}$ & $=$ \\
  \hline
\end{tabular}
} \caption{Summary of the families of graphs that either achieve
the bounds from Table~\ref{table-bounds}, or demonstrate that
there is no function bounding one parameter with another (the
latter families are bolded). Families in a bracket $[\ldots]$
demonstrate the sharpness of the bounds for graphs with at least
one edge. } \label{table-sharp}
\end{sidewaystable}

\subsection{Sharpness of the bounds}

The families of graphs used to prove sharpness of the bounds in
Table \ref{table-bounds} are summarized in Table
\ref{table-sharp}. Most of the required values for families in Table~\ref{table-sharp}
have already been established in Table~\ref{table-families}.
In fact, in Table~\ref{table-sharp}
there are only three graph families whose values have not yet been determined and are used to show the sharpness of bounds
(note that the families marked in bold letters in Table~\ref{table-sharp} are used to show
the non-existence of a function that would bound one parameter with another one).

We start with the family $S(K_{1,n})^-$, that appears in the entries (1,13) and (4,13)
of the table. Note that the corresponding bounds are $\gama(G)\le \gamaR(G)-1$ and
$\gamasettwo(G)\le 2\gamaR(G)-2$ for an arbitrary graph $G$ with edges.
Recall that the graphs $S(K_{1,n})^-$ are obtained from the subdivision
graph of the star $K_{1,n}$ by deleting a leaf. It is easy to see that $\gamma(S(K_{1,n})^-)=n$ and $\gamasettwo(S(K_{1,n})^-)=2n$
for $n\ge 3$, the main argument being that the closed neighborhoods of the $n$ vertices of degree $1$
are pairwise disjoint. To prove the sharpness of the two bounds it remains to prove the following.

\begin{claim}
For $n\ge 3$, $\gamaR(S(K_{1,n})^-)=n+1$.
\end{claim}
\begin{proof}
Note that the function $f:V(G)\to \{0,1,2\}$ that assigns $2$ to the unique
vertex $v$ of degree $n$, assigns $1$ to all vertices at distance
$2$ from $v$, and assigns $0$ to all the remaining vertices of
$S(K_{1,n})^-$, is a Roman dominating function of the graph.
$\gamaR(K_{1,n})^-)\le n+1$. On the other hand, since $\gamma(S(K_{1,n})^-)=n$,
and since $\gamma(G)\le \gamaR(G)-1$ for any graph $G$ with an edge (see Table~\ref{table-bounds}),
we infer the claimed result.
\end{proof}

We continue with the subdivision graph $S(K_{2n+1})$
of the complete graph of odd order $2n+1$, which appears in the
entries (2,3) and (2,6). Note that the bounds from the table show
that $\gamat(G)\le\frac{3\gamawtwo(G)-1}{2} \le \frac{3\gamatwo(G)-1}{2}$ for
any graph $G$.

\begin{claim}
For $n\ge 2$, $\gamawtwo(S(K_{2n+1}))=\gamatwo(S(K_{2n+1})) =2n+1$ and $\gamat(S(K_{2n+1}))= 3n+1$.
\end{claim}
\begin{proof}
Note that the vertex set of $S(K_{2n+1})$ is given by $V\cup {V\choose 2}$, where $V =V(K_{2n+1})$
and ${V\choose 2}$ are the vertices added in the subdivision of $K_{2n+1}$.
We denote by $x_{uv}$ to the vertex added in the subdivision of the edge $uv$. Clearly $S(K_{2n+1})$ is bipartite with bipartition
$\big\{V,{V\choose 2}\big\}$. On the one hand, $V$ is a $2$-dominating set of $S(K_{2n+1})$,
showing that $\gamawtwo(S(K_{2n+1}))\le \gamatwo(S(K_{2n+1}))\le 2n+1$. On the other hand, we
claim that $\gamat(S(K_{2n+1}))\ge 3n+1$. Indeed, suppose to the contrary
that $D$ is a total dominating set of $S(K_{2n+1})$ with at most $3n$
vertices. Then either $|D\cap V|\le 2n-1$ or $|D\cap {V\choose
2}|\le n$. In the former case, there exists a pair $u,v\in V$ of
distinct vertices not in $D$, and therefore $N_{S(K_{2n+1})}(x_{uv})\cap D =
\{u,v\}\cap D = \emptyset$. In the latter case, vertices of $D\cap
{V\choose 2}$ dominate at most $2n$ vertices in $V$, hence there
exists a non-dominated vertex in $V$. In either case, we obtain a
contradiction.
Therefore,
$$3n+1\le \gamat(S(K_{2n+1}))\le \frac{3\gamawtwo(S(K_{2n+1}))-1}{2}\le \frac{3(2n+1)-1}{2} = 3n+1$$
and equalities hold throughout.
\end{proof}

Finally, the only remaining entries in Table~\ref{table-sharp}
that demonstrate sharpness of the bounds and
do not follow from entries in Table~\ref{table-families}
are (13,3) and (13,6).
The bounds show that $\gamaR(G)\le 2\gamawtwo(G)-1\le 2\gamatwo(G)-1$
for any graph $G$. The sharpness is demonstrated by the family of
graphs $G_n = S(K_n^{(2)})$, which are the subdivision graphs of the
multigraphs obtained from the complete graphs by duplicating each
edge (see Fig.~\ref{fig:Kn**andSKn2} for an example).
On the one hand, the (weak) $2$-domination number of $G_n$ is
at most $n$, since assigning weight $1$ to each original vertex of
$K_n^{(2)}$ and weight $0$ to all other vertices results in a (weak)
$2$-dominating function of $G_n$ of weight $n$.
On the other hand, we will now show that the Roman domination
number of $G_n$ is at least $2n-1$.

\begin{claim}
For $n\ge 3$, $\gamaR(S(K_n^{(2)}))=2n-1$.
\end{claim}

\begin{proof}
Among all minimum weight Roman
dominating functions of $G_n$, choose one, say $f$, that minimizes
the value of $S_f$, defined as the sum of $f$-weights of all
subdivided vertices. First, we will show that $f(s) = 0$ for every
subdivided vertex $s$. Indeed, suppose for a contradiction that
$f(s) > 0$ where $s$ is a subdivided vertex of maximum $f$-weight.
Let $s'$ be the twin of $s$, that is, the vertex $s'\neq s$ such
that $N_{G_n}(s') = N_{G_n}(s)$, and let $t$ and $t'$ be the two
(common) neighbors of $s$ and $s'$ in $G_n$. We consider two
cases:
\begin{itemize}
  \item {\it Case 1: $f(s) = 2$.}

Then, $f(s')\le 1$, since otherwise a Roman dominating function
with smaller weight than $f$ could be obtained, by setting $f(s')
= 1$.

If both neighbors of $s$ have $f$-weight $0$, then we could obtain
a Roman dominating function $g$ of the same weight as $f$ and such
that $S_g<S_f$ by setting $g(s) = g(s') = 0$, $g(t) = 1$, $g(t') =
2$, and $g(u) = f(u)$ for all other vertices $u\in V(G_n)$. This
contradicts the choice of $f$.

If both neighbors of $s$ have positive $f$-weight, then a Roman
dominating function with smaller weight than $f$ could be obtained
by setting $f(s) = 1$, so this case is also impossible.

Hence, we may assume that $f(t) = 0$ and $f(t') \in \{1,2\}$. In
this case, we could obtain a Roman dominating function $g$ of the
same weight as $f$ and such that $S_g<S_f$ by setting $g(s) = 0$,
$g(t) = 2$, and $g(u) = f(u)$ for all other vertices $u\in
V(G_n)$.

  \item {\it Case 2: $f(s) = 1$.}

On the one hand, by the choice of $s$ we have
$f(s')\in\{0,1\}$. On the other hand, by the minimality of the total weight of $f$, $f(s') = 1$, for
otherwise $f(t) = 2$ or $f(t') = 2$, and so $f(s)$ could be set to
$0$ without violating the constraints of Roman domination. A Roman dominating function $g$ of weight at most that of $f$ and
such that $S_g<S_f$ can be obtained by setting $g(s) = g(s') = 0$,
$g(t) = 2$, and $g(u) = f(u)$ for all other vertices $u\in
V(G_n)$. This contradicts the choice of $f$.
\end{itemize}

Since $f(s) = 0$ for every subdivided vertex $s$, for every pair
of original vertices $t$ and $t'$, either $f(t) = 2$ or $f(t') =
2$ (or both). Hence, at most one original vertex can have weight
less than $2$. If such a vertex exists, its weight must be $1$,
hence $\gamaR(G_n) = f(V(G_n))\ge 2n-1$. By the above,
since $\gamawtwo(G)\le n$, we infer that $\gamaR(G_n) = 2n-1$.
\end{proof}

By this it is proven that all non-bold entries of Table~\ref{table-sharp}
demonstrate the sharpness of the corresponding bounds from Table~\ref{table-bounds}.

%%%%%%%%%%%%%%%%%%%%%%%%%%%%%%%%%%%%%%%%%%%
\subsection{Proofs of unboundedness}\label{ss:unbound}
For the direct proofs of unboundedness of one parameter with respect to another
one can use the families of graphs summarized in Table \ref{table-sharp}.
We will prove in this section the correctness of these examples.
As we will elaborate, some of the unboundedness proofs follow by transitivity,
using the bounds in Table \ref{table-bounds} and are summarized in Table
\ref{table-proofs}.

While the values for the star $K_{1,n}$ are easy to prove,
we can argue the nonexistence of corresponding functions only by
focusing on two parameters, notably $\gamatwo$ and $\gamatsettwo$.

\begin{prp}\label{prp:unb-gamatwo-gamatsettwo}
There is no function $f: \mathbb{N} \to \mathbb{N}$ such that
$\gamatwo(G)\le f(\gamatsettwo(G))$ for every graph $G$ admitting both parameters.
\end{prp}
\begin{proof}
It is easy to see that $\gamatwo(K_{1,n})=n$ for $n\ge 2$, since the leaves
cannot be dominated by a $2$-dominating function $f$ from the outside,
i.e., each leaf $u$ must be assigned $f(u)=1$. On the other hand,
assigning $f(v)=2$ to the central vertex, and $f(u)=2$ to one of the leaves,
results in a total $\{2\}$-dominating function of $K_{1,n}$,
thus $\gamatsettwo(K_{1,n})\le 4$.
\end{proof}

Since the parameters in columns 1-4, 9, and 13 are bounded from above by
a function of $\gamatsettwo(G)$ (see Table~\ref{table-bounds}),
we derive that $K_{1,n}$ is also an example for these invariants compared with $\gamatwo$.
In addition, since the parameters
$\gamaxtwo,\rgamatwo,\rgamaxtwo$ are bounded from below by $\gamatwo$
(see the diagram on Fig.~\ref{fig:Hasse})
we infer from both observations that the entries $(i,j)$ from the subtable
$\{6,7,10,11\}\times \{1,2,3,4,5,9,13\}$ of Table~\ref{table-bounds},
are correct. That is, the family of stars $K_{1,n}$ shows  that there does not
exist an upper bound on a parameter $\rho$ in terms of a function of another parameter $\rho'$
for all the corresponding pairs $(\rho,\rho')$.
Note that this family cannot be used to obtain similar conclusions
also for row $12$, that is, for the parameter $\rgamatxtwo(G)$:
this parameter is not finite on the family of stars.

Consider now the graphs $Q_n$, for $n\ge 3$, which can be obtained from
the multigraph $K_2^{(n)}$ by replacing each edge with a path $P_4$ (cf.~Fig.~\ref{fig:SK3nQn}).
These graphs will be used for the proofs of unboundedness
in the row 8 of Tables~\ref{table-bounds}~and~\ref{table-families}, concerning the parameter $\gamatxtwo$.

\begin{prp}\label{prp:unb-gamatxtwo-gamatsettwo}
There is no function $f: \mathbb{N} \to \mathbb{N}$ such that
$\gamatxtwo(G)\le f(\gamatsettwo(G))$ for every graph $G$ admitting both parameters.
\end{prp}
\begin{proof}
It is easy to see that $\gamatxtwo(Q_n)=2n+2$. Indeed,
if $f:V(G)\to\{0,1\}$ is a total double dominating
function, then for every vertex of degree 2 both its neighbors
must receive $f$-value 1. This implies that all vertices
of $Q_n$ must receive value 1. To see that $\gamatsettwo(Q_n)\le 8$ for $n\ge 3$,
consider the function $f$ assigning $2$ to both vertices of
degree $n$, and $2$ to one of the neighbors of each of these two vertices.
\end{proof}

Since the parameters $\gama,\gamat,\gamawtwo,\gamasettwo,\rgamawtwo,\gamaR$
are all bounded from above by $\gamatsettwo$, we infer from Proposition~\ref{prp:unb-gamatxtwo-gamatsettwo}
that all entries $(i,j)$ from $\{8\}\times\{1,2,3,4,5,9,13\}$ of Table~\ref{table-bounds} are correct.
In fact, the family of graphs $Q_n$ demonstrates the nonexistence of a
function $f$ bounding the corresponding parameters with $f(\gamatxtwo(G))$.

Recall that the graph $S(K_3^{(n)})$ is the subdivision graph of the multigraph
$K_3^{(n)}$ (the multigraph obtained from $K_3$ by adding $n-1$
parallel edges between each pair of vertices; cf.~Fig.~\ref{fig:SK3nQn}).

\begin{prp}\label{prp:unb-rgamatwo-gamatxtwo}
There is no function $f: \mathbb{N} \to \mathbb{N}$ such that
$\rgamatwo(G)\le f(\gamatxtwo(G))$ for every graph $G$ admitting both parameters.
\end{prp}
\begin{proof}
The result follows from the correctness of the entries (Y,8) and (Y,10)
in Table~\ref{table-families}. To see this consider the function
$f:V(S(K_3^{(n)}))\to\{0,1\}$, which assigns $1$ exactly to the three vertices
of degree $2n$ and to three vertices of degree 2, one from each subdivided parallel edge.
Then $f$ is clearly a total double dominating function of the graph, with total
weight $6$. On the other hand, note that for $n\ge 3$,
every rainbow $2$-dominating function of $S(K_3^{(n)})$
that assigns the empty set to a vertex of degree $2n$ is of total weight
at least $2n$. Furthermore, in every rainbow $2$-dominating function of $S(K_3^{(n)})$
that assigns a non-empty set to each vertex of degree $2n$,
at least two vertices of degree $2n$ receive the same value, hence
the neighbors $v$ of these two vertices must receive a non-empty value.
The above arguments imply that $\rgamatwo(S(K_3^{(n)}))\ge n+3$.
\end{proof}

Since $\gamatwo$ and $\gamaxtwo$ are bounded from above by $\gamatxtwo(G)$,
we derive that $S(K_3^{(n)})$ is also an example for these invariants
with respect to $\rgamatwo$. In addition, the parameter $\rgamaxtwo$ is
bounded from below by $\rgamatwo$ (see the Hasse diagram on Fig.~\ref{fig:Hasse}),
which together with the previous observation implies
that the entries $(i,j)$ from the subtable
$\{10,11\}\times \{6,7,8\}$ of Table~\ref{table-bounds}, are correct.
In fact, the family of graphs $S(K_3^{(n)})$ can be used to demonstrate that there does not
exist an upper bound on a parameter $\rho$ in terms of a function of another parameter $\rho'$
for all the corresponding pairs $(\rho,\rho')$.

\begin{prp}\label{prp:unb-rgamatxtwo-rgamaxtwo}
There is no function $f: \mathbb{N} \to \mathbb{N}$ such that
$\rgamatxtwo(G)\le f(\rgamaxtwo(G))$, for every graph $G$
admitting both parameters.
\end{prp}

\begin{proof}
Recall that $T_n$, for $n \in \mathbb{N}$, is the graph whose vertex set
is $\{v_1, \dots, v_n,$ $w_1, \dots, w_n,$ $s_1,$ $s_2,$ $s_3,$
$t_1, \dots, t_5\}$ and such that $s_1,s_2,s_3$ induce a triangle,
$t_1, \dots, t_5$ induce $C_5$, $s_1$ and $s_3$ are adjacent to
$v_i$ for every $1 \leq i \leq n$, $t_1$ and $t_5$ are adjacent to
$w_i$ for every $1 \leq i \leq n$, $v_iw_i \in E(T_n)$ for every
$1 \leq i \leq n$, and there are no other edges (see Fig.~\ref{fig:Tn}).

Consider $h$ defined as $h(s_1)=h(t_1)=h(t_4)=\{a\}$,
$h(s_3)=h(t_3)=h(t_5)=\{b\}$, and $h(v) = \emptyset$ for every
other $v \in V(T_n)$. It can be easily checked that $h$ is a
rainbow double domination function of $T_n$, as for every vertex
$v\in V(T_n)$, it holds $h_\cup(N[v]) = \{a,b\}$. So
$\rgamaxtwo(T_n) \le 6$, and indeed it can be seen that it holds
by equality, because $s_2$, $t_3$ and $w_1$ have disjoint closed
neighborhoods.

As for the rainbow total double domination number, where
$h_\cup(N(v)) = \{a,b\}$ is required for every vertex $v \in
V(T_n)$, first note that if a vertex $v$ has degree two in a graph
then its two neighbors have to be labeled with different labels.
This implies that in a rainbow total double domination function
$h$ of $T_n$, $h(t_1) \neq h(t_3)$ and $h(t_5) \neq h(t_3)$, hence
$h(t_1)=h(t_5)$. So the vertices $w_1, \dots, w_n$ are missing one
label on their open neighborhoods, therefore vertices $v_1, \dots,
v_n$ have to have a nonempty label, implying $\rgamatxtwo(T_n) \ge
n$. Indeed, a rainbow total double domination function $h$ of
$T_n$ can be defined as
$h(s_1)=h(s_2)=h(t_1)=h(t_2)=h(t_5)=\{a\}$,
$h(s_3)=h(t_3)=h(t_4)=h(w_1)=\{b\}$, $h(v_i)=\{b\}$ for every $1
\leq i \leq n$, and $h(w_i) = \emptyset$ for every $2 \leq i \leq
n$. It is not hard to see that its weight is minimum possible, so
$\rgamatxtwo(T_n) = n + 9$.
\end{proof}

Since all the parameters (except of course for $\rgamatxtwo$) in graphs $G$
are bounded by a function of $\rgamatwo(G)$ we infer that the entries
in the row $12$ of Table~\ref{table-bounds} are correct. In fact,
the family of graphs $T_n$ demonstrates the nonexistence of a
function $f$ bounding any of the other parameters with $f(\rgamatxtwo(G))$.

\section{Algorithmic and complexity issues}\label{sec:algo}

We now discuss the algorithmic and complexity consequences of the bounds obtained in Section~\ref{sec:comparison} for corresponding optimization problems. More specifically, we obtain new results regarding the existence of approximation algorithms for the studied invariants,
matched with tight or almost tight inapproximability bounds, which hold even in the class of split graphs.

Recall that an algorithm ${\cal A}$ for a minimization problem $\Pi$ is said to be a {\em $c$-approximation algorithm} (where $c\ge 1$) if  it runs in polynomial time and for every instance $I$ of $\Pi$, we have ${\cal A}(I)\le c\cdot {\it OPT}(I)$, where ${\cal A}(I)$ is the value of the solution produced by ${\cal A}$, given $I$, and ${\it OPT}(I)$ is the optimal solution value, given $I$. (For more details on complexity and approximation, we refer to~\cite{MR1851303,MR1734026}.)  Given a graph $G$, let $\rho(G)$ denote the optimal value of any of the minimization parameters studied in this paper (e.g., the domination number of $G$, the rainbow total double domination number of $G$, etc.). The corresponding optimization problem is the following problem: Given a graph $G$, compute the value of $\rho(G)$.
In the case of a $c$-approximation algorithm for the above problem, we also require that for each instance $G$ not only an approximation to the optimal value but also a feasible solution to the problem is computed achieving value at most $c\cdot \rho(G)$.
Note that in the problems relating to any of the parameters considered in this paper, a feasible solution is a function $f$ with domain $V$, whose value equals the total weight $f(V)$ (see Section~\ref{sec:definitions}).

First we recall a simple (folklore) observation that can be useful for transferring both lower and upper bounds regarding (in)approximability of minimization problems. In order to keep the notation as simple as possible, we keep the presentation of the result confined to the parameters defined in Section~\ref{sec:definitions}, however, the same result clearly applies more generally. For the sake of completeness, we include the simple proof.

\begin{prp}\label{prp:reduction}
Let $\rho$ and $\rho'$ be any two graph invariants defined in Section~\ref{sec:definitions} and
let ${\cal G}$ be a class of graphs such that there exist constants $c_1,c_2>0$ such that
for all $G\in {\cal G}$, we have $c_1\cdot \rho(G)\le \rho'(G)\le c_2\cdot \rho(G)\,.$
Suppose furthermore that there exists a polynomial time algorithm that for a given
graph $G\in {\cal G}$ and a feasible solution $f$ to $\rho$, computes a feasible solution $f'$
to $\rho'$ with $f'(V(G))\le c_2\cdot f(V(G))$. Then, for every $c\ge 1$, if there is a $c$-approximation algorithm for $\rho$ on graphs in ${\cal G}$, then there is $(cc_2/c_1)$-approximation algorithm for $\rho'$ on graphs in ${\cal G}$.
\end{prp}

\begin{proof}
Let ${\cal A}$ be a $c$-approximation algorithm for $\rho$ on graphs in ${\cal G}$.
Consider the following algorithm for $\rho'$ on graphs in ${\cal G}$:
\begin{enumerate}
  \item Given a graph $G\in {\cal G}$, run ${\cal A}$ on $G$ and let $f_{\cal A}$ be the solution produced by ${\cal A}$.
  \item Compute a feasible solution $f'$ to $\rho'$ with $f'(V(G))\le c_2\cdot f_{\cal A}(V(G))$ using the algorithm that exists by assumption.
  \item Return $f'$.
\end{enumerate}
Since ${\cal A}$ is a $c$-approximation algorithm for $\rho$ on graphs in ${\cal G}$, we have $ f_{\cal A}(V(G))\le c\rho(G)$. It follows that
$f'(V(G)) \le  c_2\cdot f_{\cal A}(V(G))\le c_2\cdot c\rho(G)\le (c_2c/c_1)\rho'(G)\,,$ where the last inequality follows from $c_1\rho(G)\le \rho'(G)$. As the algorithm clearly runs in polynomial time, it is a $(c_2c/c_1)$-approximation algorithm for $\rho'$ for graphs in ${\cal G}$.
\end{proof}

Note that all the bounds from Table~\ref{table-bounds} are of the form $\rho'(G)\le c\rho(G) - d$ for some constants $c\ge 1$ and $d\ge 0$, hence
they immediately imply bounds of the form $\rho'(G)\le c\rho(G)$ (for some constant $c\ge 1$).
Furthermore, it follows from the proofs of the bounds that all the translations between parameters involving bounds summarized in Table~\ref{table-bounds} can be efficiently constructed, in the sense that if $\rho'(G)\le c\rho(G)$ is a bound following from bounds in Table~\ref{table-bounds}, then there is a polynomial time algorithm that, given a graph $G = (V,E)$ and a feasible solution $f$ to $\rho$, computes a feasible solution $f'$ to $\rho'$ with $f'(V)\le c\cdot f(V)$.

\subsection{Lower bounds}

\begin{sloppypar}
Several hardness and inapproximability results for variants of domination considered in the paper are already known in the literature. We list here only the strongest results known and an earliest available proof for each of them, making no attempt to survey the literature regarding hardness of the problems in various graph classes -- with the single exception of the class of split graphs, which naturally appears in many of the underlying proofs.
A graph $G = (V,E)$ is said to be {\em split} if it admits a {\em split partition}, that is, a pair $(C,I)$ such that $C$ is a clique in $G$, $I$ is an independent set in $G$, $C\cup I = V$, and $C\cap I = \emptyset$. Split graphs were introduced by F\"oldes and Hammer in~\cite{MR0463041}, where several characterizations were also given.
\end{sloppypar}

\vbox{\begin{sloppypar}
\begin{thm}[combining results from~\cite{MR2457654,MR3027964,MR2024938,MR2960327}]\label{thm:inapprox-lower-bound-known}
For every $\rho\in \{\gama,\gamat,\gamatwo,$ $\gamaxtwo,\gamatxtwo\}$ and every $\epsilon >0$, there is no polynomial time algorithm approximating $\rho$ for $n$-vertex split graphs without isolated vertices within a factor of \hbox{$(1-\epsilon)\ln n$}, unless $\NP\subseteq {\sf DTIME}(n^{O(\log\log n)})$.
\end{thm}
\end{sloppypar}}

\begin{sloppypar}
The statement of Theorem~\ref{thm:inapprox-lower-bound-known} was proved:
\begin{enumerate}[(i)]
        \item for domination and total domination ($\gama,\gamat$) by Chleb{\'{\i}}k and Chleb{\'{\i}}kov{\'a} in~\cite{MR2457654},
        \item for $2$-domination  ($\gamatwo$) by Cicalese et al.~\cite{MR3027964} (in the more general context of $k$-domination),
        \item for double domination  ($\gamaxtwo$) by Klasing and Laforest~\cite{MR2024938} (in the more general context of $k$-tuple domination),
         \item for total double domination ($\gamatxtwo$) independently by Pradhan~\cite{MR2960327} and by Cicalese et al.~\cite{MR3027964} (in both cases in the more general context of $k$-tuple domination).
\end{enumerate}
Only the results by Chleb{\'{\i}}k and Chleb{\'{\i}}kov{\'a} were mentioned explicitly for split graphs. However, since the corresponding reductions from~\cite{MR3027964,MR2024938,MR2960327} are performed from either domination or total domination by simply adding a number of universal vertices to the input graph, all of the above results also hold for split graphs.
\end{sloppypar}

{The basis of the inapproximability results from~\cite{MR2960327,MR2457654,MR2024938,MR3027964} summarized in Theorem~\ref{thm:inapprox-lower-bound-known} is the analogous result due to Feige for the well-known {\sc Set Cover} problem: Given a set system $(S,{\cal F})$ where $S$ is a finite set (also called a {\it ground set}) and ${\cal F}$ is a family (multiset) of subsets of $S$, find a smallest {\it set cover} of ${\cal F}$, that is, a sub-collection ${\cal F'}\subseteq {\cal F}$ such that $\bigcup{\cal F}' = S$ (that is, such that every element of $S$ appears in some member of ${\cal F}'$).}

\begin{thm}[Feige~\cite{MR1675095}]
For every $\epsilon >0$, there is no polynomial time algorithm approximating {\sc Set Cover} within a factor of \hbox{$(1-\epsilon)\ln n$} (where $n$ is the size of the ground set), unless $\NP\subseteq {\sf DTIME}(n^{O(\log\log n)})$.
\end{thm}

{
In~$2014$, Dinur and Steurer improved Feige's inapproximability result by weakening the hypothesis to ${\sf P} \neq {\sf NP}$.

\begin{thm}[Dinur and Steurer~\cite{MR3238990}]\label{thm:Dinur-Steurer}
For every $\epsilon >0$, there is no polynomial time algorithm approximating {\sc Set Cover} within a factor of \hbox{$(1-\epsilon)\ln n$}, unless ${\sf P} ={\sf NP}$.
\end{thm}

An essential fact in proving the bounds from Theorem~\ref{thm:inapprox-lower-bound-known} is that the instances of {\sc Set Cover} arising in Feige's construction are such that $\ln(|S|+|{\cal F}|)\approx \ln |S|$, that is, the ratio $\ln(|S|+|{\cal F}|)/\ln |S|$ can be assumed to be arbitrarily close to~$1$. This is also true for the instances of arising in the construction proving Theorem~\ref{thm:Dinur-Steurer}.
Consequently, Theorem~\ref{thm:inapprox-lower-bound-known} can be improved as follows:

\begin{sloppypar}
\begin{thm}\label{thm:inapprox-lower-bound-improved}
For every $\rho\in \{\gama,\gamat,\gamatwo,$ $\gamaxtwo,\gamatxtwo\}$ and every $\epsilon >0$, there is no polynomial time algorithm approximating $\rho$ for $n$-vertex split graphs without isolated vertices within a factor of \hbox{$(1-\epsilon)\ln n$}, unless ${\sf P} = \NP$.
\end{thm}
\end{sloppypar}}

In particular, the above results imply that the decision variants of the corresponding optimization problems are \NP-complete.

\begin{sloppypar}
We are not aware of inapproximability results for any of the invariants \hbox{$\rho\in \{\gamaR, \gamawtwo, \gamasettwo, \gamatsettwo, \rgamatwo, \rgamawtwo, \rgamaxtwo, \rgamatxtwo\}$.} (Recall that invariants $\rgamatwo$, $\rgamaxtwo$, and $\rgamatxtwo$ are, to the best of our knowledge, considered for the first time in this paper.) The following \NP-completeness results for some of these parameters are available in the literature:
\begin{itemize}
  \item The \NP-completeness of Roman domination ($\displaystyle{\gamma}\!\:_R$) was proved by Dreyer in~\cite{MR2701485}. (The problem was already claimed to be \NP-complete in~\cite{cdhh-2004}, referring to a private communication with A.A.~McRae.)
  \item The weak $2$-domination ($\gamawtwo$) and the rainbow weak $2$-domination ($\rgamawtwo$) problems were proved \NP-complete by Bre\v{s}ar and Kraner \v{S}umenjak in~\cite{bks-2007}.
  \item The \NP-completeness of $\{2\}$-domination (${\displaystyle\gamma}\!\:_{{\{2\}}}$) was proved
  by Gairing et al.~in~\cite{GHKM2003} (in the more general context of $\{k\}$-domination).
  \item The \NP-completeness of rainbow double domination ($\rgamaxtwo(G)$; for graphs without isolated vertices) follows from the analogous result due to Hedetniemi et al.~\cite{MR2500476} for disjoint domination (cf.~Proposition~\ref{prp:disjoint-domination}).
\end{itemize}
We are not aware of any published hardness results about total $\{2\}$-domination ($\gamasettwo$).
\end{sloppypar}

In the rest of this subsection, we strengthen the above \NP-completeness results by showing that all the domination parameters studied in this paper, except for the rainbow total double domination number, admit an inapproximability bound of the form $\Omega(\ln n)$ for $n$-vertex split graphs, {unless ${\sf P} = {\sf NP}$.} Before doing that, we show that for the rainbow total double domination number ($\rgamatxtwo$; recall that this is the topmost parameter in the diagram in Fig.~\ref{fig:Hasse-quotient}), the situation is even worse.
We say that a graph $G$ is $\widetilde{\displaystyle\gamma}\!\:_{t\!{\scriptstyle \times} \! 2}$-feasible
if $\rgamatxtwo(G)$ is finite (cf.~Proposition~\ref{prp:rgamatxtwo} on p.~\pageref{prp:rgamatxtwo}).

\begin{thm}
There is no polynomially computable function $f$ such that there exists an $f(n)$-approximation algorithm for
rainbow total double domination on $n$-vertex
$\widetilde{\displaystyle\gamma}\!\:_{t\!{\scriptstyle \times} \! 2}$-feasible
split graphs, unless
 ${\sf P} = {\sf NP}$.
\end{thm}

\begin{proof}
Suppose for a contradiction that there exists a polynomially computable function $f$ such that there exists an $f(n)$-approximation algorithm for
rainbow total double domination on $n$-vertex
$\widetilde{\displaystyle\gamma}\!\:_{t\!{\scriptstyle \times} \! 2}$-feasible
split graphs.
We will show that this implies  ${\sf P} = {\sf NP}$, by designing a polynomial time algorithm for the \NP-complete
 {\sc Hypergraph $2$-Colorability} problem~\cite{MR519066}, which asks whether a given hypergraph is $2$-colorable.
 A {\it hypergraph} ${\cal H}$ is a pair $(V,{\cal E})$ where $V$ is a finite set and ${\cal E}$ is a set of subsets of $V$.
 A hypergraph is said to be {\it $2$-colorable} if its vertex set $V$ admits a partition into two independent sets $A$ and $B$, where a set $X\subseteq V$ is {\it independent} if it does not contain any hyperedge $e\in {\cal E}$.
 We may assume that $|A|\ge 2$ and $|B|\ge 2$ in every partition as above since otherwise the problem can be solved in polynomial time.

Given an input ${\cal H} = (V,{\cal E})$ to the {\sc Hypergraph $2$-Colorability} problem,
construct the split graph $G = (V',E)$ with split partition $(C,I)$ where
$C = V$, $I = {\cal E}$, and there is an edge in $G$ between $v\in C$ and $e\in I$  if and only if
$v\in e$. Clearly, $G$ can be constructed from ${\cal H}$ in polynomial time.

We claim that ${\cal H}$ is $2$-colorable if and only if $G$ is $\widetilde{\displaystyle\gamma}\!\:_{t\!{\scriptstyle \times} \! 2}$-feasible.
First, suppose that ${\cal H}$ is $2$-colorable, and let $\{A,B\}$  be a partition of $V$ into two independent sets.
Then, the function $g:V(G)\to \{\emptyset,\{a\},\{b\}\}$ defined by
$$g(v) = \left\{
          \begin{array}{ll}
            \{a\}, & \hbox{if $v\in A$;} \\
            \{b\}, & \hbox{if $v\in B$;} \\
            \emptyset, & \hbox{otherwise.}
          \end{array}
        \right.$$
 is a rainbow total double dominating function of $G$. Indeed, the assumption
 $|A|\ge 2$ and $|B|\ge 2$ implies that
 $g_{\cup}(N(v)) = \{a,b\}$ for all $v\in C$, while the fact that
 $A$ and $B$ are both independent in ${\cal H}$ implies that $g$ also dominates vertices in $I$.
 It follows that $G$ is   $\widetilde{\displaystyle\gamma}\!\:_{t\!{\scriptstyle \times} \! 2}$-feasible.
Conversely, suppose that $G$ is $\widetilde{\displaystyle\gamma}\!\:_{t\!{\scriptstyle \times} \! 2}$-feasible, with a
rainbow total double dominating function  $g:V(G)\to \{\{a\},\{b\},\emptyset\}$.
Modify $g$ if necessary by setting $g(v) = \{a\}$ for every $v\in C$ with $g(v) = \emptyset$; clearly, the so obtained function is still
a rainbow total double dominating function of $G$. Moreover, the sets
 $A = \{v\in V: g(v) = a\}$ and
 $B = \{v\in V: g(v) = b\}$ form a partition of $V$, the vertex set of ${\cal H}$.
Since  $g_{\cup}(N(v)) = \{a,b\}$ for all $v\in I$, each of the sets $A$ and $B$ is independent in ${\cal H}$, and thus ${\cal H}$ is $2$-colorable.

Now, let $n = |V'|$, and let ${\cal A}$ be an $f(n)$-approximation algorithm for rainbow total double domination on $n$-vertex $\widetilde{\displaystyle\gamma}\!\:_{t\!{\scriptstyle \times} \! 2}$-feasible split graphs. We know that ${\cal A}$ computes a rainbow total double dominating function on
$\widetilde{\displaystyle\gamma}\!\:_{t\!{\scriptstyle \times} \! 2}$-feasible
 split graphs, but if the input graph is not of this form, there is no guarantee about what ${\cal A}$ computes or whether it even halts. By definition ${\cal A}$ runs in polynomial time on $n$-vertex $\widetilde{\displaystyle\gamma}\!\:_{t\!{\scriptstyle \times} \! 2}$-feasible
 split graphs, say its running time is bounded by a polynomial $p(n)$.

The polynomial time algorithm that decides whether ${\cal H}$ is $2$-colorable goes as follows.
\begin{enumerate}
  \item Construct the split graph $G$ as specified above.
  \item Compute $n = |V(G)|$ and $f(n)$, and let ${\cal A}$ be an $f(n)$-approximation algorithm
 for rainbow total double domination on $n$-vertex $\widetilde{\displaystyle\gamma}\!\:_{t\!{\scriptstyle \times} \! 2}$-feasible split graphs.
  \item Run ${\cal A}$ on $G$ for at most $p(n)$ steps.
  \item If ${\cal A}$ did not compute anything, then $G$ is not
  $\widetilde{\displaystyle\gamma}\!\:_{t\!{\scriptstyle \times} \! 2}$-feasible.
We conclude that ${\cal H}$ is not $2$-colorable.
  \item If ${\cal A}$ computed something, then check whether what it computed is a rainbow total double dominating function on $G$.

   If it is, then $G$ is $\widetilde{\displaystyle\gamma}\!\:_{t\!{\scriptstyle \times} \! 2}$-feasible, and we conclude that ${\cal H}$ is  $2$-colorable.
   (In this case we also have that the total weight of the computed function is at most $f(n)\rgamatxtwo(G)$, but we will not need this fact.)

   If it is not, then $G$ is not $\widetilde{\displaystyle\gamma}\!\:_{t\!{\scriptstyle \times} \! 2}$-feasible, and we conclude that ${\cal H}$ is not $2$-colorable.
\end{enumerate}
It is clear that the algorithm runs in polynomial time. Its correctness follows from the correctness of ${\cal A}$ and from the fact that
${\cal H}$ is $2$-colorable if and only if $G$ is
$\widetilde{\displaystyle\gamma}\!\:_{t\!{\scriptstyle \times} \! 2}$-feasible.
Thus, the above algorithm efficiently solves the \NP-complete
 {\sc Hypergraph $2$-Colorability} problem, implying that
${\sf P} = {\sf NP}$. This completes the proof.\end{proof}

\medskip
\begin{sloppypar}
We now turn out attention to the remaining parameters. Proposition~\ref{prp:reduction} and the discussion following it show that
in order to prove an inapproximability bound of the form $\Omega(\ln n)$ for each of the remaining considered parameters, namely
\hbox{$\rho\in \{\gamawtwo, \gamasettwo, \gamatsettwo, \gamaR, \rgamatwo, \rgamawtwo, \rgamaxtwo\}$,} it suffices to show an inapproximability bound of the same type for just one parameter in each of the bottom three equivalence classes
in the diagram of Fig.~\ref{fig:Hasse-quotient}. As mentioned above, such bounds already exist, even for the class of split graphs, for any
$\rho\in \{\gama,\gamat, \gamatwo,\gamaxtwo, \gamatxtwo\}$, which takes care of the invariants appearing in the
bottom two equivalence classes in Fig.~\ref{fig:Hasse-quotient}. We summarize this in the following theorem.
\end{sloppypar}

\begin{thm}\label{thm:inapprox-lower-bound}
For every $\rho\in \{\gamawtwo, \gamasettwo, \gamatsettwo, \gamaR, \rgamawtwo, \rgamasettwo, \rgamatsettwo\}$ and every $\epsilon >0$, there is no polynomial time algorithm approximating $\rho$ for $n$-vertex split graphs without isolated vertices within a factor of \hbox{$(1/2-\epsilon)\ln n$}, {unless ${\sf P} = {\sf NP}$.}
\end{thm}

\begin{proof}
Recall from Section~\ref{sec:comparison} that for every graph $G$ without isolated vertices, we have
$\gama(G)\le \gamawtwo(G)\le \rgamawtwo(G)\le\gamaR(G)\le \rgamasettwo(G) = 2\gama(G)$ and
$\gamat(G)\le \gamasettwo(G)\le \gamatsettwo(G)\le \rgamatsettwo(G) = 2\gamat(G)\,.$
Thus, the theorem follows from the inapproximability bound for domination (resp., total domination), see
{Theorem~\ref{thm:inapprox-lower-bound-improved}}, the above inequalities, and Proposition~\ref{prp:reduction}.
We prove the statement formally only for the weak $2$-domination number ($\gamawtwo$); the
proofs for the other parameters are analogous.

Let ${\cal G}$ be the class of split graphs without isolated vertices and suppose that
there is some $\epsilon >0$ such that there is a polynomial time algorithm approximating the weak $2$-domination number
on $n$-vertex graphs in ${\cal G}$ within a factor of
\hbox{$(1/2-\epsilon)\ln n$}. For every graph $G$, we have
$\frac{1}{2}\gamawtwo(G)\le \gama(G)\le \gamawtwo(G)$. Moreover,
for every weak $2$-dominating function $f$ of $G$, the set $\{v\in V(G): f(v)>0\}$ is a dominating function
of $G$ of weight at most $f(V(G))$. Therefore, Proposition~\ref{prp:reduction} applies with $c_1 = 1/2$, $c_2 = 1$, and hence
there is a polynomial time algorithm approximating the domination number on $n$-vertex graphs in ${\cal G}$ within a factor of
\hbox{$(1-2\epsilon)\ln n$}. {By Theorem~\ref{thm:inapprox-lower-bound-improved},
this is only possible if ${\sf P} = {\sf NP}$.}
\end{proof}

We also explicitly state the following consequence of Theorem~\ref{thm:inapprox-lower-bound} for total $\{2\}$-domination ($\gamatsettwo$),
which does not seem to be yet available in the literature.

\begin{cor}
The decision version of the total $\{2\}$-domination problem is \NP-complete.
\end{cor}

The remaining equivalence class from Fig.~\ref{fig:Hasse-quotient} contains two parameters, namely rainbow $2$-domination ($\rgamatwo$) and rainbow double domination ($\rgamaxtwo$). { Using a reduction from {\sc Set Cover}, we now prove the inapproximability bounds for the rainbow $2$-domination ($\rgamatwo$) and the rainbow double domination ($\rgamaxtwo$) problems in split graphs. As discussed above, it would suffice to prove a bound for only one of the two parameters. We give a direct proof for both parameters, since with almost no additional work, we save a multiplicative factor of $2$ in one of the two bounds compared to the bounds we would obtain using the above approach.}

\begin{thm}\label{thm:inapprox}
For every $\rho\in \{\rgamatwo, \rgamaxtwo\}$ and every $\epsilon >0$, there is no polynomial time \hbox{$(1-\epsilon)\ln n$}-approximation algorithm for computing $\rho$ on $n$-vertex split graphs, unless {{\sf P} = {\sf NP}}.
\end{thm}

\begin{sloppypar}
\begin{proof}
Fix $\rho\in \{\rgamatwo, \rgamaxtwo\}$ and suppose for some $\epsilon >0$, there is a polynomial time \hbox{$(1-\epsilon)\ln n$}-approximation algorithm, say ${\cal A}$, for computing $\rho$ on $n$-vertex split graphs.

Let $J= (S,{\cal F})$ be an instance to the {\sc Set Cover} problem. First, note that we may assume that
\begin{equation}\label{eq:set-cover}
\ln 3 + \ln(|S|+|{\cal F}|) \le (1+\epsilon/2)\ln(|S|+|{\cal F}|)\le (1+\epsilon)\ln|S|\,.
\end{equation}
Indeed, if the first inequality above is violated, then $\ln(|S|+|{\cal F}|)$ is bounded by $2\ln 3/ \epsilon$ and the problem can be solved in constant time. The second inequality follows from the fact that the ratio $\ln(|S|+|{\cal F}|)/(\ln |S|)$ can be made arbitrarily close to $1$
{(as remarked right after Theorem~\ref{thm:Dinur-Steurer}).}

Consider the split graph $G_J = (V,E)$ with split partition $(C,I)$ where
$C = A\cup B$ with $A =  \{a_F\mid F\in {\cal F}\}$,
$B= \{b_F\mid F\in {\cal F}\}$,
$I = S_1\cup S_2\cup S_3$ with $S_j = \{s_j: s\in S\}$ for $j\in \{1,2,3\}$,
and there is an edge between
$F_i$
for $i\in \{a,b\}$, $F\in {\cal F}$
and $s_j\in S_j$ for $j\in \{1,2,3\}$,
$s\in S$ if and only if $s\in F$.

Graph $G_J$ has $3|S|+2|{\cal F}|\le 3(|S|+|{\cal F}|)\le |S|^{1+\epsilon}$ vertices and can be computed in polynomial time from $J$.
Let ${\it OPT}$ denote the minimum size of a set cover for ${\cal F}$. First, we prove the following claim.

\medskip
\noindent{\bf Claim:} $\rgamatwo(G_J) = \rgamaxtwo(G_J)= 2\cdot{\it OPT}$.

\medskip
\noindent{\it Proof of claim:} The inequality $\rgamatwo(G_J) \le  \rgamaxtwo(G_J)$ always holds (see Table~\ref{table-bounds}).
Thus, it remains to prove $\rgamaxtwo(G_J)\le 2\cdot{\it OPT}$ and $2\cdot{\it OPT}\le \rgamatwo(G_J)$.

We first prove that $\rgamaxtwo(G_J)\le 2\cdot{\it OPT}$.
First, let ${\cal F'}$ be a minimum set cover for ${\cal F}$.
Consider the function $f:V(G_J)\to\{\emptyset,\{a\},\{b\}\}$ defined as follows:
$$f(v) = \left\{
           \begin{array}{ll}
             \{a\}, & \hbox{if $v = a_F\in A$ and $F\in {\cal F}'$;} \\
             \{b\}, & \hbox{if $v = b_F\in B$ and $F\in {\cal F}'$;} \\
             \emptyset, & \hbox{otherwise.}
           \end{array}
         \right.
$$
Clearly, $f(V(G_J)) = 2|{\cal F}'| = 2\cdot{\it OPT}$. Thus, to prove that $\rgamaxtwo(G_J)\le  2\cdot{\it OPT}$, it suffices to check that
$f$ is a rainbow double dominating function of $G_J$, that is, that
$f_\cup(N[v]) = \{a,b\}$ holds for all $v\in V$.
If $v\in I$, then $v= s_j$ for some $s\in S$ and some $j\in \{1,2,3\}$. There exists some $F\in {\cal F}'$ with $s\in F$.
This implies that $a_F$ and $b_F$ are adjacent to $s_j$ in $G_J$, and by construction these two vertices are labeled $\{a\}$ and $\{b\}$, respectively.
If $v\in C$, then we have either $f(v)  = \{a\}$ or $f(v)  = \{b\}$ or $f(v)  = \emptyset$.
If $f(v)  = \{a\}$ then since $C$ is a clique, any vertex
$b_F$ with $F\in {\cal F}'$ is a neighbor of $v$ labeled $\{b\}$.
The case when $f(v)  = \{b\}$ is symmetric. Finally, if
$f(v) = \emptyset$, then we similarly observe that $v$ is adjacent
to both a vertex of the form $a_F$ and a vertex of the form $b_F$ (with $F\in {\cal F}'$).
It follows that $f$ is a rainbow $\{2\}$-dominating function of $G_J$, which implies $\rgamaxtwo(G_J)\le  2\cdot{\it OPT}$.

Now, we prove that ${\it OPT}\le \rgamatwo(G_J)/2$.
Let $f:V(G_J)\to\{\emptyset,\{a\},\{b\}\}$ be a minimum
rainbow $2$-dominating function.
We therefore have $f_\cup(N(v)) = \{a,b\}$ for all $v\in V(G_J)$ with $f(v) = \emptyset$.
First, we will show that we have
$f(v) = \emptyset$ for all $v\in I$.
Suppose for a contradiction that $f(v) \neq\emptyset$ for some $v\in I$.
By minimality of $f$, the function obtained by relabeling
$v$ to $\emptyset$ is not a rainbow $2$-dominating function of $G_J$, which implies that
$v$ does not have both labels $\{a\}$ and $\{b\}$ in its neighborhood. Assume that
$a\not\in f_{\cup}N(v)$ (the other case is symmetric). Let $s\in S$ and $j\in \{1,2,3\}$ be such that $v = s_j$.
Then $a\not\in f_{\cup}N(u)$ for all $u\in \{s_1,s_2,s_3\}$, which implies that
$f(s_j) \neq\emptyset$ for all $j\in \{1,2,3\}$. Let $F\in {\cal F}$ such that $s\in F$, and consider the function $f'$ obtained from $f$ by relabeling
each of $s_j$ to $\emptyset$, and by setting $f'(a_F) = \{a\}$ and
$f'(b_F) = \{b\}$ (and leaving all other values unchanged).
It is easy to see that $f'$ is a rainbow $2$-dominating function of smaller total weight than $f$. This is a contradiction with
the minimality of $f$ and proves that
$f(v) = \emptyset$ for all $v\in I$.
This assumption implies that every $v\in I$ has both labels $\{a\}$ and $\{b\}$ in its neighborhood.

The minimality of $f$ implies that for every $F\in {\cal F}$,
at most one of $a_F$ and $b_F$ gets label $\{a\}$. (If both
$a_F$ and $b_F$ would get label $\{a\}$, then replacing one of them with $\emptyset$ would result in a
rainbow $2$-dominating function of $G_J$ of smaller total weight than $f$.)
Similarly, at most one of $a_F$ and $b_F$ gets label $\{b\}$.
Also, by the symmetry of the construction, we may assume
that if one of
$a_F$ and $b_F$ gets label $\{a\}$, then $f(a_F) = \{a\}$, and
that if one of $a_F$ and $b_F$ gets label $\{b\}$, then
$f(b_F) = \{b\}$.
Thus, $A' = \{v\in C: f(v)= \{a\}\}$ and $B' = \{v\in C: f(v)= \{b\}\}$
satisfy $A'\subseteq A$ and $B'\subseteq B$.
Without loss of generality assume that $|A'|\le |B'|$.
We claim that ${\cal F}' = \{F\in {\cal F}: a_F\in A'\}$ is a set cover of ${\cal F}$.
Indeed, if $s\in S$, then the fact that $s_1\in I$ and every
vertex in $I$ has label $\{a\}$ in its neighborhood implies that there is a vertex $a_F\in N(s_1)$ such that
$f(a_F) = \{a\}$, in other words $a_F\in A'$, which implies that $s\in F$ (since $a_F\in N(s_1)$)
and $F\in {\cal F}'$ (since $a_F\in A'$).
Since ${\cal F}'$ is a set cover of ${\cal F}$, it follows that
$${\it OPT}\le |{\cal F}'| = |A'|\le\frac{|A'|+|B'|}{2}  = \frac{f(V(G_J))}{2}= \frac{\rgamatwo(G_J)}{2}\,.$$

This completes the proof of the claim. \hfill $\blacktriangle$

\medskip
Now we can complete the proof of the theorem. Recall that ${\cal A}$ is a polynomial time \hbox{$(1-\epsilon)\ln n$}-approximation algorithm for computing $\rho$ on $n$-vertex split graphs. Using ${\cal A}$, we can design an approximation algorithm for {\sc Set Cover}, transforming an instance $J  = (S,{\cal F})$ to the split graph $G_J$, computing an approximate solution $f$ to $\rho$ on $G_J$, and returning the corresponding set cover ${\cal F}'$ obtained from $f$ as in the above proof of the claim.
Letting $n = |V(G_J)|$, we can bound the size of ${\cal F}'$ from above as
$$
\begin{array}{rcll}
|{\cal F}'| &\le& f(V(G_J))/2 & \textrm{(by the above proof of the claim)}\\
 &\le& (1-\epsilon)(\ln n)\rho(G_J)/2 & \textrm{(since $f$ was computed using the}\\
 &&  & \phantom{(}\textrm{${(1-\epsilon)\ln n}$-approximation algorithm ${\cal A}$)}\\
 &\le& (1-\epsilon)\ln(3(|S|+|{\cal F}|)) {\it OPT} & \textrm{(since
 $n \le 3(|S|+|{\cal F}|)$ and $\rho(G_J) = 2\cdot{\it OPT}$)}\\
 &\le& (1-\epsilon) (1+\epsilon) (\ln|S|) {\it OPT} & \textrm{(by \eqref{eq:set-cover})}\\
 &\le & (1-\epsilon^2) (\ln|S|) {\it OPT}\,, &\\
\end{array}
$$
Therefore, there exists a polynomial time algorithm that computes a
$(1-\epsilon^2)\ln |S|$-approximation to {\sc Set Cover}. By
{Theorem~\ref{thm:Dinur-Steurer}, this is only possible if ${\sf P} = {\sf NP}$}.
\end{proof}
\end{sloppypar}

\subsection{Upper bounds}

The following theorem summarizes the upper bounds on approximability of domination parameters considered in this paper available in the literature:

\begin{thm}[combining results from~\cite{MR686528,MR3027964,MR2024938,MR2960327}]\label{thm:approx-upper-bound-known}
~\begin{enumerate}
  \item For each $\rho\in \{\gama,\gamaxtwo\}$, there is a $(\ln(\Delta(G)+1)+1)$-approximation algorithm for $\rho$.
  \item For each $\rho\in \{\gamat,\gamatxtwo\}$, there is a $(\ln(\Delta(G))+1)$-approximation algorithm for $\rho$.
  \item For each $\rho\in \{\gamatwo,\gamawtwo\}$, there is a $(\ln(\Delta(G)+2)+1)$-approximation algorithm $\rho$.
\end{enumerate}
\end{thm}

The result for domination and total domination follows from the fact that these two problems can be easily modeled as special cases of {\sc Set Cover}. It is well known that a simple greedy algorithm for {\sc Set Cover} produces a solution that is always within a factor $\ln \Delta+1$ of the optimum, where $\Delta$ is the maximum size of a set in ${\cal F}$~\cite{MR686528}. As proved independently by Dobson~\cite{MR686528} and by Klasing and Laforest~\cite{MR2024938}, the same is true for the more general problem in which the task is to find a minimum size subcollection ${\cal F}'\subseteq {\cal F}$ such that every element $s$ appears in at least $k$ sets in ${\cal F}'$ (Dobson's result is in fact more general: each vertex can have a different coverage requirement). In turn, this implies the above-mentioned approximation results for double domination ($\gamaxtwo$) and total double domination ($\gamatxtwo$); see~\cite{MR2024938,MR2960327,MR3027964}.
The result for $2$-domination ($\gamatwo$) can be obtained with a straightforward modification
of the proof of~\cite[Theorem 3]{MR3027964}. That result gives an approximation algorithm for the more general problem
called {\it vector domination} (in which one seeks a small subset $S$ of vertices of a graph such that
any vertex outside $S$ has at least a prescribed number of neighbors in $S$), using a reduction to the so-called
{\sc Minimum Submodular Cover} problem and applying a result of Wolsey~\cite{MR708153}.

Without trying to optimize the obtained approximation ratios, let us simply note that Theorem~\ref{thm:approx-upper-bound-known} and a similar approach to that used in the proof of Theorem~\ref{thm:inapprox-lower-bound} implies the following result.

\begin{thm}\label{thm:approx-upper-bound}
For every $\rho\in \{\gamasettwo, \gamatsettwo, \rgamawtwo, \gamaR, \rgamasettwo, \rgamatsettwo\}$, there is a $2(\ln(\Delta(G)+2)+1)$-approximation algorithm for $\rho$.
\end{thm}

To the best of our knowledge, these are the first results regarding approximation algorithms for any of these parameters.
Development of approximation algorithms for rainbow $2$-domination and rainbow double domination
remains an open question.

\bigskip
We conclude with the following related questions, which we leave for future research:
\begin{itemize}
  \item Can the factors $1/2-\epsilon$ in the inapproximability bounds from Theorem~\ref{thm:inapprox-lower-bound} be improved to $1-\epsilon$? Possible approaches to this question include a development of direct reductions from {\sc Set Cover} and a study of the inequalities relating the relevant parameters in the class of split graphs.
 \item Can the approximation ratios given by Theorem~\ref{thm:approx-upper-bound} be further improved?
 \item The only known inapproximability bound for the rainbow $2$-domination and rainbow double domination problems are those given by Theorem~\ref{thm:inapprox}, and no nontrivial approximation algorithms for these two problems are known.
It would be interesting to settle the \hbox{(in-)approximability} status of these two problems.
The case of the rainbow double domination number $\rgamaxtwo(G)$ of a graph $G$ without isolated vertices is particularly interesting,
because the parameter coincides with the previously studied disjoint domination number $\gamma\gamma(G)$ (cf.~Proposition~\ref{prp:disjoint-domination}).
\end{itemize}

%%%%%%%%%%%%%%%%%%%%%%%%%%%%%%%% REFERENCES
\subsection*{Acknowledgements}

The work for this paper was done in the framework of a bilateral project between Argentina and Slovenia, financed by the Slovenian Research Agency (BI-AR/$12$--$14$--$012$, BI-AR/$12$--$14$--$013$) and MINCyT, Argentina (SLO/11/12 and SLO/11/13). F.\ Bonomo, L.N.\ Grippo, and M.D.\ Safe were partially supported by UBACyT Grant 20020130100808BA, CONICET PIP 112-201201-00450CO, and ANPCyT PICT-2012-1324 (Argentina).
F. Bonomo was partially supported by ANPCyT PICT-2015-2218 (Argentina). L.N.\ Grippo and M.D.\ Safe were partially supported by PIO CONICET UNGS 144-20140100027-CO (Argentina).
B.\ Bre\v sar was supported in part by the Slovenian Research Agency (ARRS) under the grant P$1$-$0297$. Research of M.\ Milani\v c was supported in part by the Slovenian Research Agency (I0-0035, research program P$1$-$0285$ and research projects N$1$-$0032$, J$1$-$5433$, J$1$-$6720,$ J$1$-$6743$, and J$1$-$7051$).

\bibliographystyle{abbrv}
\bibliography{2-domination}

\end{document}